\documentclass[a4paper,11pt]{amsart}
\usepackage{amsmath, amsfonts, amssymb, amsthm, amscd}
\usepackage{graphicx}
\usepackage{enumerate}
\usepackage[colorlinks=true, linkcolor=blue, citecolor=blue]{hyperref}
\usepackage{url}
\usepackage{color}
\usepackage[utf8]{inputenc}
\usepackage{tikz}
\usetikzlibrary{arrows,patterns} 
\usepackage[a4paper,scale={0.72,0.74},marginratio={1:1},footskip=7mm,headsep=10mm]{geometry}

\usepackage[normalem]{ulem}

\numberwithin{equation}{section}

\newtheorem{theorem}{Theorem}[section]
\newtheorem{lemma}[theorem]{Lemma}
\newtheorem{proposition}[theorem]{Proposition}

\newcommand{\bbE}{{\ensuremath{\mathbb E}} }

\newcommand{\bbN}{{\ensuremath{\mathbb N}} }

\newcommand{\bbP}{{\ensuremath{\mathbb P}} }

\newcommand{\bbR}{{\ensuremath{\mathbb R}} }

\newcommand{\bbZ}{{\ensuremath{\mathbb Z}} }

\newcommand{\cB}{{\ensuremath{\mathcal B}} }

\newcommand{\cN}{{\ensuremath{\mathcal N}} }

\newcommand{\cP}{{\ensuremath{\mathcal P}} }

\newcommand{\cR}{{\ensuremath{\mathcal R}} }
\newcommand{\cS}{{\ensuremath{\mathcal S}} }

\newcommand{\cZ}{{\ensuremath{\mathcal Z}} }

\newcommand{\ga}{\alpha}
\newcommand{\gb}{\beta}
\newcommand{\gga}{\gamma}

\newcommand{\gd}{\delta}
\newcommand{\gD}{\Delta}
\newcommand{\gep}{\varepsilon} 
\newcommand{\gz}{\zeta}
\newcommand{\gt}{\theta}
\newcommand{\gk}{\kappa}
\newcommand{\gl}{\lambda}
\newcommand{\gL}{\Lambda}

\newcommand{\gs}{\sigma}

\newcommand{\gO}{\Omega}
\newcommand{\8}{\infty}

\renewcommand{\tilde}{\widetilde}          % wider `tilde'
\DeclareMathSymbol{\leqslant}{\mathalpha}{AMSa}{"36} % nicer `smaller or equal'
\DeclareMathSymbol{\geqslant}{\mathalpha}{AMSa}{"3E} % nicer `larger or equal'
\DeclareMathSymbol{\eset}{\mathalpha}{AMSb}{"3F}     % nicer `emptyset'
\newcommand{\dd}{\text{\rm d}}             % a straight d for differentials
\newcommand{\sumtwo}[2]{\sum_{\substack{#1 \\ #2}}} % sum with 2 lines
\newcommand{\R}{\mathbb{R}}

\newcommand{\Z}{\mathbb{Z}}
\newcommand{\N}{\mathbb{N}}

\newcommand{\PEfont}{\mathrm}
\DeclareMathOperator{\var}{\ensuremath{\PEfont Var}}

\newcommand\bP{\ensuremath{\mathrm{P}}}

\newcommand\bE{\ensuremath{\mathrm{E}}}

\DeclareMathOperator\argmin{argmin}
\renewcommand{\epsilon}{\varepsilon}

\newcommand{\ind}{{\sf 1}}

\newcommand{\card}{\mathrm{card}}

\newenvironment{myenumerate}{
\renewcommand{\theenumi}{\arabic{enumi}}
\renewcommand{\labelenumi}{{\rm(\theenumi)}}
\begin{list}{\labelenumi}
{
\setlength{\itemsep}{0.4em}
\setlength{\topsep}{0.5em}
\setlength\leftmargin{2.45em}
\setlength\labelwidth{2.05em}
\setlength{\labelsep}{0.4em}
\usecounter{enumi}
}
}
{\end{list}
}

\renewenvironment{enumerate}{
\begin{myenumerate}}
{\end{myenumerate}}

\newcommand{\beq}{\begin{equation}}
\newcommand{\eeq}{\end{equation}}
\newcommand{\ba}{\begin{aligned}}
\newcommand{\ea}{\end{aligned}}

\newcommand{\bin}{\mathrm{Bin}}
\newcommand{\ber}{\mathrm{Ber}}

\begin{document}

\title[The simple random walk in power-law renewal traps]{A limit theorem for the survival probability of a simple random walk among power-law renewal traps}

\author[J. Poisat]{Julien Poisat}
\address[J. Poisat]{Universit\'e Paris-Dauphine, CNRS, UMR [7534], CEREMADE, PSL Research University, 75016 Paris, France}
\email{poisat@ceremade.dauphine.fr}

\author[F. Simenhaus]{Fran\c cois Simenhaus}
\address[F. Simenhaus]{Universit\'e Paris-Dauphine, CNRS, UMR [7534], CEREMADE, 75016 Paris, France}
\email{simenhaus@ceremade.dauphine.fr}

\begin{abstract} We consider a one-dimensional simple random walk surviving among a field of static soft traps : each time it meets a trap the walk is killed with probability $1-e^{-\gb}$, where $\gb$ is a positive and fixed parameter. The positions of the traps are sampled independently from the walk and according to a renewal process. The increments between consecutive traps, or gaps, are assumed to have a power-law decaying tail with exponent $\gamma > 0$. We prove convergence in law for the properly rescaled logarithm of the quenched survival probability as time goes to infinity. The normalization exponent is $\gga/(\gga+2)$, while the limiting law writes as a variational formula with both universal and non-universal features. The latter involves (i) a Poisson point process that emerges as the universal scaling limit of the properly rescaled gaps and (ii) a function of the parameter $\gb$ that we call {\it asymptotic cost of crossing per trap} and that may, in principle, depend on the details of the gap distribution. Our proof {suggests} a {confinement strategy} of the walk in a single large gap.
This model may also be seen as a $(1+1)$-directed polymer among many repulsive interfaces, in which case $\gb$ corresponds to the strength of repulsion, the survival probability to the partition function and its logarithm to the finite-volume free energy. 
%Along the way we prove a stochastic monotonicity property for the hitting time of the killed random walk with respect to the non-killed one, that could be of interest in other contexts, see Proposition \ref{prop:FKG}.
\end{abstract}

\thanks{\textit{Acknowledgements}. JP and FS are supported by the ANR/FNS-16-CE93-0003 grant MALIN. JP is also supported by the ANR-17-CE40-0032 grant SWiWS. FS is also supported by the ANR-15-CE40-0020-01 grant LSD}

\maketitle

{\it Disclaimer.} In this paper we denote by $\bbN$ the set of positive integers and $\bbN_0 = \bbN \cup \{0\}$. The letter $C$ is used for the constants whose values are irrelevant and may change from line to line.\\

\section{Introduction and model}
\label{sec:intro}

\par We consider a one-dimensional simple random walk surviving among a random field of static soft traps~: each time it meets a trap the walk is killed with probability $1-e^{-\gb}$ or survives with probability $e^{-\gb}$, where $\gb$ is a positive and fixed parameter, see Section~\ref{sec:rwrt} for a precise definition. The increments between consecutive traps, or gaps, are assumed to be mutually independent and independent from the walk, with a power-law decaying tail. This is what we refer to as {\it power-law renewal traps} in the title, see Section~\ref{sec:traps} for a precise definition. To be more precise, we deal with the {\it quenched} version of the model, meaning that the positions of the traps are frozen and the survival probability is computed only with respect to the law of the random walk.
Our main result, Theorem~\ref{thm0}, states a convergence in law for the properly rescaled logarithm of the quenched survival probability, seen as a random variable with respect to the field of traps, as time goes to infinity. The limiting law writes as a variational formula involving (i) a Poisson point process that emerges as the universal scaling limit of the properly rescaled gaps and (ii) a function of the parameter $\gb$ that we call {\it asymptotic cost of crossing per trap} and that may, in principle, depend on the details of the gap distribution, see the definition of $\lambda(\beta)$ in Proposition~\ref{pr:lambda.beta}. Even if we offer no path statement for the walk conditioned on surviving, our proof strongly suggests a {\it confinement strategy} according to which the walk remains in a large gap with an appropriate scale. Path localization could be considered as future work.
\par As we will see in Section~\ref{sec:polymer}, our model may also been seen as a $(1+1)$-directed polymer among many repulsive interfaces, in which case $\gb$ corresponds to the strength of repulsion. We will take advantage of this connection by using and revisiting some estimates obtained by Caravenna and Pétrélis~\cite{CP09b,CP09} in the case of periodic traps. We point out that the logarithm of the survival probability is the finite-volume free energy of the corresponding polymer model.\\

\par {\it Outline.} Section~\ref{sec:intro} contains the mathematical definition of the model and a discussion on the relation with other models, such as the directed polymer among multiple interfaces. Section~\ref{sec:result} contains the statement of Theorem~\ref{thm0}, which is our main result. At the beginning of the section we introduce several mathematical objects which are necessary for the understanding of the theorem. Comments and related open questions are listed at the end of Section~\ref{sec:result}. Sections~\ref{sec:KT} to~\ref{sec:pot} constitute the proof of the theorem. Key tools are gathered in Section~\ref{sec:KT}. The rest of the proof is split into a lower bound part (Section~\ref{sec:LB}), an upper bound part (Section~\ref{sec:UB}) and a conclusion (Section~\ref{sec:pot}). The more technical proofs are deferred to an appendix.

\subsection{A random walk in $\bbZ$ among soft traps}
\label{sec:rwrt}

We consider $S=(S_n)_{n\ge 0}$ a simple random walk on $\bbZ$ in presence of traps. We recall that the increments $(S_n - S_{n-1})_{n\ge 1}$ are independent and identically distributed (i.i.d.) random variables which are uniformly distributed on $\{-1,1\}$, and we shall write $\bP_x$ for the law of the walk started at $S_0 =x$, for $x\in\bbZ$, with the notational simplification $\bP_0 =\bP$. The positions of the traps are integers and they will be denoted by $\tau = \{\tau_n\}_{n\ge 0}$ with $\tau_0=0$. The increments $(\tau_{n+1} - \tau_n)_{n\ge0}$ may be referred to as {\it gaps}.
\par Let $\gb>0$ be a parameter of the model. Informally, each time the walk meets a trap, it is killed with probability $1-e^{-\gb}$ or survives with probability $e^{-\gb}$ independently from the past. The traps are called {\it soft}, by opposition to {\it hard} traps, because the walk has a positive chance to survive when sitting on a trap. For a precise mathematical definition let us first introduce $(\gt_n)_{n\ge 1}$, the clock process recording the times when the random walk visits $\tau$, that is
\beq
\gt_0 = 0,\quad \gt_{n+1} = \inf\{k > \gt_n \colon S_k \in \tau\},\quad n\ge 0.
\eeq
We enlarge the probability space so as to include a $\N$-valued geometric random variable $\cN$ with success parameter $1-e^{-\gb}$. This plays the role of the clock that kills the walk after its $\cN$-th meeting with the set of traps. We now define $\gs$ the death time of the walk by
\beq
\sigma = \gt_{\cN}.
\eeq
Note that our probability law now depends on the parameter $\beta$. We shall write $\bP^\beta_x$ when we want to stress this dependence or omit the superscript when no confusion is possible. Again we may write $\bP^\beta$ instead of $\bP_0^\beta$. We also point out that $\gs$ depends on $\tau$ through $\theta$ even if it is not explicit in the notations.\\
\par The {\it hitting times} of the walk are defined by
\beq
H_x = \inf\{n\geq 1\colon S_n = x \}, \qquad x\in\bbZ,
\eeq
and
\beq
H_{\Z^-} = \inf\{n\geq 1\colon S_n \leq 0 \}.
\eeq
\par In this paper we study the limiting behaviour of the probability that the walk survives up to time $n\in \bbN$, as $n$ gets large. For convenience, we consider walks that do not visit $\Z^-=-\N_0$. This extra condition allows to consider traps that are indexed by $\bbN_0$ instead of $\bbZ$ and does not hide anything deep nor change the main idea of the paper. Thus, our {\it survival probability} writes
\beq
\label{eq:defZn}
Z_n = \bP^{\gb}(\sigma \wedge H_{\Z^-} > n), \qquad n\in\bbN.
\eeq
We stress again that $Z_n$ is a function of the environment of traps $(\tau_n)_{n\geq 0}$.

\subsection{(De)pinning of a $(1+1)$-directed polymer by multiple interfaces} 
\label{sec:polymer}

\par By integrating on $\cN$ in~\eqref{eq:defZn}, we obtain
\beq
Z_n = \bE\Big[\exp\Big(-\gb \sum_{k=1}^n \ind_{\{S_k \in \tau\}}\Big) \ind_{\{H_{\Z^-}>n\}} \Big].
\eeq
This expression links the survival probability to a certain polymer model from statistical mechanics. More precisely, the expression above is the partition function of a $(1+1)$-directed polymer above an impenetrable wall and among many repulsive interfaces. Here, $(k,S_k)_{0\le k\le n}$ plays the role of a polymer with $n$ monomers and the parameter $n$, which is initially a time parameter, becomes the size of the polymer, see Figure~\ref{polymer}. Whenever the polymer touches one of the interfaces, located at the levels $\tau = \{\tau_n\}_{n\ge 0}$, it is penalized by a factor $e^{-\gb}$. Finally, the event $\{H_{\Z^-} > n\}$ reflects the presence of a hard wall at level $0$.

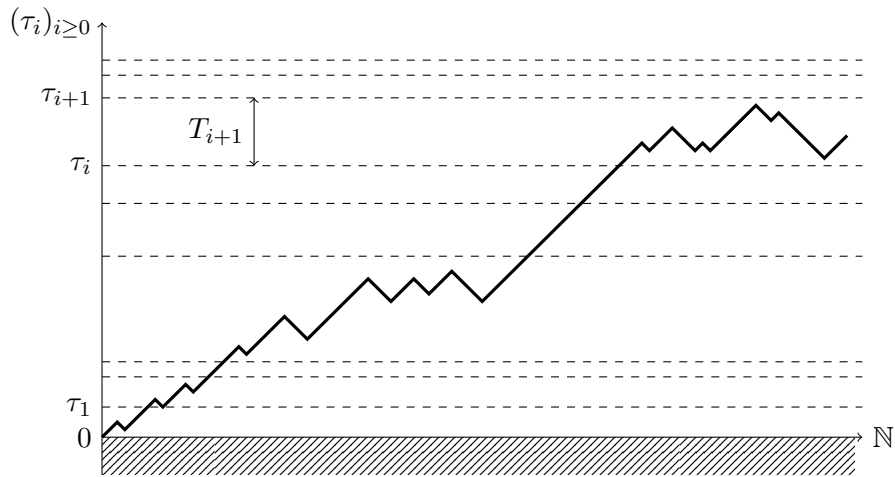
\begin{figure}[!h]
\centering
 \begin{tikzpicture}[scale=1]
 \draw[->] (-0,0) -- (10,0);
 \draw [->] (0,-0.5) -- (0,5.5);
 \draw (0,0) node[left] {$0$};
\draw (0,5.5) node[left] {$(\tau_i)_{i\geq 0}$};
\draw (10,0) node[right] {$\N$};
\draw (0,0.4) node[left] {$\tau_1$};
\draw (0,3.6) node[left] {$\tau_i$};
\draw (0,4.5) node[left] {$\tau_{i+1}$};
\draw (2,4.05) node[left] {$T_{i+1}$};
\draw [<->] (2,3.6) -- (2,4.5);
\draw [dashed] (0,1) -- (10,1);
\draw  [dashed] (0,0.4) -- (10,0.4);
\draw [dashed] (0,0.8) -- (10,0.8);
\draw  [dashed] (0,2.4) -- (10,2.4);
\draw [dashed](0,3.1) -- (10,3.1);
\draw [dashed](0,3.6) -- (10,3.6);
\draw [dashed](0,4.5) -- (10,4.5);
\draw[dashed] (0,5) -- (10,5);
\draw [dashed](0,4.8) -- (10,4.8);
\fill[pattern=north east lines] (0,0) -- (9.9,0) -- (9.9,-0.5)-- (0,-0.5) ;
\draw [very thick] (0,0) -- (0.1,0.1)--(0.2,0.2)--(0.3,0.1)--(0.4,0.2)--(0.5,0.3)--(0.6,0.4)--(0.7,0.5)--(0.8,0.4)--(0.9,0.5)--(1.0,0.6)--(1.1,0.7)--(1.2,0.6)--(1.8,1.2)--(1.9,1.1)--(2.4,1.6)--(2.6,1.4)--(2.7,1.3)--(2.9,1.5)--(3.5,2.1)--(3.8,1.8)--(4.1,2.1)--(4.3,1.9)--(4.6,2.2)--(5,1.8)--(5.4,2.2)--(7.1,3.9)--(7.2,3.8)--(7.5,4.1)--(7.8,3.8)--(7.9,3.9)--(8,3.8)--(8.6,4.4)--(8.8,4.2)--(8.9,4.3)--(9.5,3.7)--(9.8,4);
\end{tikzpicture}
\caption{Example of a polymer among repulsive interfaces. The dashed lines correspond to the interfaces, the thick one to the polymer and the shaded area to the hard wall.}
\label{polymer}
\end{figure}

\par There is a rich literature about the pinning phenomenon in polymer models, for which we refer to~\cite{dH09, Gia07, Gia11}. In general, the parameter $\gb$ can have any sign (thus $\gb<0$ corresponds to the attractive case with our notations) and in the simplest case of a unique interface at level zero without wall, one observes a localization transition at $\beta=0$. The case $\beta<0$ is known as the \textit{repulsive} or \textit{delocalized} phase as the polymer typically touches the interface only finitely many times. The case $\beta >0$ is called instead the \textit{attractive} or \textit{localized} phase as the polymer typically touches the interface a positive fraction of time. Caravenna and Pétrélis studied the case of multiple interfaces periodically located at levels $t_n\bbZ$, where the period $t_n\ge 1$ depends on the size of the polymer $n$, both in the attractive~\cite{CP09} and repulsive case~\cite{CP09b}. In contrast, the positions of our interfaces are random, but do not vary in $n$. However, the size of the {\it relevant gaps} does depend on $n$, which explains why we use (and extend) some estimates from~\cite{CP09b}. Note also the difference between our results, since Caravenna and Pétrélis obtained results at the level of paths, which means information on the (de)localization of the path of the random walk under the polymer measure
\beq
\frac{\dd \bP_n^{\gb}}{\dd \bP} \propto \exp\Big(-\gb \sum_{k=1}^n \ind_{\{S_k \in \tau\}}\Big),
\eeq
(where $\propto$ means ``proportional to'') while this is not the purpose of the present paper, see Comment $8$ in Section~\ref{subsec:c}.

\subsection{Assumption on the traps}
\label{sec:traps}
We now put a probability measure $\bbP$ on the environment of traps.
We denote by $T_k = \tau_k - \tau_{k-1}$, for $k\in\bbN$, the increments, that is the size of the intervals between two consecutive traps, which we call {\it gaps}. We assume that, under $\bbP$, $\tau$ is a discrete renewal process, that is the $(T_k)$'s are i.i.d. $\bbN$-valued random variables. We further assume that $\tau_0=0$ and that the increments have a power-tail distribution:
\beq
\label{eq:tail.ass}
\bbP(T_1 = n) \sim c_\tau\ n^{-(1+\gamma)},\qquad \gamma>0,\quad n\to\infty,
\eeq
where $c_\tau$ is a positive constant.
We recall the following standard limit theorem, see e.g.\ Petrov~\cite[Theorem 14, p. 91]{Petrov}.
\begin{proposition}
\label{thm:limit.stable}
If $\gamma\in(0,1]$, the sequence $(\tau_n/ n^{1/\gamma})_{n\ge 1}$ converges in law to a (totally skewed to the right) $\gamma$-stable random variable with scale parameter $c_\tau$. If $\gamma>1$ then $(\tau_n/n)_{n\ge 1}$ converges almost-surely to $\bbE(T_1)$.
\end{proposition}

\subsection{Parabolic Anderson model with a correlated potential} 
Our model is also connected to a discrete analogue of the parabolic Anderson model (PAM) with potential $V(x) = -\gb \ind_{\{x\in \tau\}} - \8 \ind_{\{x \le 0\}}$, that is the heat equation with random potential $V$, 
\beq
\partial_t u=\Delta u+ V u.
\eeq
 There is a rich literature and intense activity around the PAM. We refer to König~\cite{Ko16} for a recent survey on this topic. Note that the potential is usually chosen as a sequence of random variables that are independent in the space parameter. In contrast, our potential exhibits long-range spatial correlations, that is one of the research direction suggested in~\cite[Section 7.2]{Ko16}. For a review of potentials in discrete and continuous space, we refer to ~\cite[Section 1.5]{Ko16}.

\par Let us end this section with a reference to the classical monograph by Sznitman~\cite{Sz98} on random motions in random media. Chapter 6 is of particular interest to us as it highlights the link with directed polymers in the presence of columnar defects and introduces the concept of pinning effect of quenched path measures, see again Comment $8$ in Section~\ref{subsec:c}.

\section{Results}
\label{sec:result}

In this section we first introduce the various tools needed to state our main theorem. We first prove the existence of the asymptotic exponential cost for the walker to cross a new trap without being killed (Section \ref{subsec:acct}). Then, in Section \ref{subsec:coe}, we consider the environment as a point process and investigate some properties of its asymptotic limit that is a Poisson point process with explicit intensity measure. We finally state our result in Section \ref{subsec:sor} and conclude with a few comments in Section \ref{subsec:c}.
\subsection{Asymptotic cost of crossing traps}
\label{subsec:acct}

We first recall without proof the following standard proposition:
\begin{proposition}
\label{pr:rw_exit.dist}
For all $t\in\bbN$, 
\beq
\bP_1(H_t < H_0) = 1/t.
\eeq
\end{proposition}

Let us define, for $\ell\ge 1$, the random variable (with respect to $\tau$)
\beq
\label{def:gl.ell.gb}
\lambda(\ell,\gb) = -\frac{1}{\ell} \log \bP^{\gb}(H_{\tau_\ell} < H_{0} \wedge \sigma).
\eeq
Note that in the above definition $H_{0}$ could be safely replaced by $H_{\Z^-}$.
The next proposition gives the existence of an {\it asymptotic cost of crossing per trap}, which will play a crucial role in our theorem.
\begin{proposition}
\label{pr:lambda.beta}
For all $\gb >0$ there exists a positive constant $\gl(\gb) = \gl(\gb,\bbP)$ such that, $\bbP$-a.s.\ and in $L_1(\bbP)$,
\beq
\lim_{\ell\to\8}  \lambda(\ell,\gb) = \gl(\gb),
\eeq
with
\beq
\label{eq:encadr.lambda}
0 \le \gl(\gb) - \gb \le \bbE(\log T_1)+\log 2.
\eeq
\end{proposition}

Note that $\log T_1$ is integrable because of~\eqref{eq:tail.ass}.
\begin{proof}[Proof of Proposition~\ref{pr:lambda.beta}]
Let us define a collection of random variables indexed by:
\beq
\cZ(i,j) = - \log \bP_{\tau_i}(H_{\tau_j} < H_{\tau_i} \wedge \gs), \qquad 0\le i < j.
\eeq
Thus, we are interested in the limit of $(\cZ(0,\ell)/\ell)_{\ell\ge 1}$. Let $1\le i<j<k$. To go from $\tau_i$ to $\tau_k$ without being killed, one strategy is to go first from $\tau_i$ to $\tau_j$ and survive, then from $\tau_j$ to $\tau_k$ and survive without coming back to $\tau_j$, which, by the Markov property, leads to the inequality
\beq
\label{eq:subadd}
\cZ(i,k) \le \cZ(i,j) + \cZ(j,k).
\eeq
By stationarity of the sequence $(T_k)_{k\ge 1}$, $\cZ(i,j)$ has the same law as $\cZ(0,j-i)$. Moreover, by Proposition \ref{pr:rw_exit.dist},
\beq
\cZ(i-1,i) \le \gb + \log T_i + \log 2,\qquad i\in\bbN,
\eeq
and one gets by iterating~\eqref{eq:subadd}
\beq
\gb\ell \le \cZ(0,\ell) \le (\gb+\log 2)\ell + \sum_{1\le i\le \ell} \log T_i
\eeq
(the lower bound follows simply from the fact that surviving to a newly visited trap costs at least $e^{-\gb}$). We may now conclude with Kingman's sub-additive ergodic theorem (see Theorem 7.4.1 in \cite{Durrett}) and the law of large numbers.
\end{proof}

\begin{proposition}\label{lem:continuity_lambda}
The function $\beta\to\lambda(\beta)$ is continuous on $\R_+^*$.
\end{proposition}
\begin{proof}[Proof of Proposition~\ref{lem:continuity_lambda}]
We prove that the function is concave on $(0,+\8)$: as it is finite, it implies continuity. We observe by integrating over $\cN$ that
\beq
 \lambda(\ell, \beta)=-\frac{1}{\ell}\log\bE\Big(\exp\Big(-\beta \sum_{k=1}^{H_{\tau_\ell}}1_{\{S_k\in \tau\}} \Big)  \ind_{\{H_{\tau_\ell} < H_{0}\}}\Big).
\eeq
A basic interchange theorem allows us to write
\beq
\partial_{\beta}^2 \lambda(\ell,\beta)= -\frac{1}{\ell}\var_{\tilde \bP_\gb}\Big(\sum_{k=1}^{H_{\tau_\ell}}1_{\{S_k\in \tau\}}\Big) \le 0,
\eeq
where $\tilde \bP_\gb$ is absolutely continuous with respect to $\bP$, with Radon-Nikodym derivative:
\beq
\frac{\dd \tilde \bP_\gb}{\dd \bP} = \exp\Big( - \gb \sum_{k=1}^{H_{\tau_\ell}}1_{\{S_k\in \tau\}} \Big) \ind_{\{H_{\tau_\ell} < H_{0}\}}.
\eeq
We deduce thereby that $\beta \mapsto \lambda(\ell,\beta)$ is concave. Therefore $\gl$ is concave as the almost-sure limit of concave functions.
\end{proof}

\subsection{Convergence of the environment}
\label{subsec:coe}
In this section we recall a few elements of point processes, see~\cite{Res} for more background on this theory.
\par Define the quadrant $E:=[0,+\8)\times (0,+\8)$ and consider $\mathcal{E}$ the Borel $\sigma-$algebra on $E$. We say a measure $\mu$ on $(E,\mathcal{E})$ to be a point measure if $\mu$ can be written
\beq
\label{eq:mp}
\mu:=\sum_{i=1}^{+\8}\delta_{z_i}
\eeq
where, for any $a$ in $E$, $\delta_a$ denotes the Dirac measure in $a$ and $(z_i)_{i\geq 1}$ is a family of points in $E$. If $\mu$ can be written as in \eqref{eq:mp} we say that \textit{$z_i \in \mu$} ($i\geq 1$) even if we should say that $z_i$ is in the support of $\mu$.  
\par We call $M_p(E)$ the set of all Radon point measures on $E$, that are the point measures $\mu$ such that $\mu(K)<+\8$ for all compact sets $K\subset E$. We endow $M_p(E)$ with the $\sigma-$algebra $\mathcal{M}_p(E)$ defined as the smallest $\sigma-$algebra that makes applications $\mu \mapsto \mu(F)$ measurable for all $F \in \mathcal{E}$.
Let $C^+_K(E)$ be the set of continuous non-negative functions on $E$ with compact support. A sequence $(\mu_n)_{n\geq 1}$ in $M_p(E)$ is said to converge vaguely to $\mu$, which we note $\mu_n \stackrel{v}{\to} \mu$, if for any $f$ in $C^+_K(E)$
\beq
\int f\ \dd\mu_n \to  \int f\ \dd\mu,\qquad n\to +\8.
\eeq
This provides a topology on $M_p(E)$ that turns out to be metrisable, separable, and complete. 
In this context, a sequence of probability measures $(P_n)_{n\geq 1}$ on $\mathcal{M}_p(E)$ converges weakly to $P$, which we note $P_n \stackrel{w}{\to} P$, if for every $\Theta$ vaguely continuous and bounded on  $M_p(E)$,
\beq
\int \Theta\ \dd P_n \to  \int \Theta\ \dd P,\qquad n\to +\8.
\eeq

\par We now come back to our context. For $n\geq 1$ we define
\beq
\ba
&(X^n_i,Y^n_i) := \Big(\frac{i-1}{n}, \frac{T_i}{n^{1/\gamma}}\Big) \qquad \textrm{for all } i\geq 1,\\
&\textrm{and }\qquad \Pi_n =\sum_{i=1}^{+\infty}\delta_{(X_i^n,Y_i^n)}.
\ea
\eeq
We observe that $\Pi_n$ is a random variable that takes values in $M_p(E)$.
Recall the definition of $c_\tau$ in \eqref{eq:tail.ass} and define $\Pi$ to be a Poisson point process on $E$ with intensity 
\beq
\label{eq:def_p}
p:=\dd x \otimes \frac{c_{\tau}\gamma}{y^{\gamma+1}}\ \dd y, 
\eeq
that is, for all finite families of disjoint events $(A_i)_{1\leq i \leq n} \in \mathcal{E}^n$, $(\Pi(A_i))_{1\leq i \leq n}$ are independent Poisson random variables with respective parameters $(p(A_i))_{1\leq i \leq n}$.
\begin{proposition}
\label{prop:convergenceP}
It holds that
$
\Pi_n \stackrel{w}{\rightarrow} \Pi.
$
\end{proposition}

\begin{proof}
The proof follows from \cite[Proposition~3.21 p.~154]{Res} once noticed that for all $y>0$, the sequence $(n \bbP(T_1>y n^{1/\gamma}))_{n\geq 1}$ converges to $c_\tau / y^\gamma$ when $n$ goes to infinity.
\end{proof}

We define for any $\lambda>0$,
\beq
\label{eq:psi0}
\begin{array}{ccccc}
 \psi^\lambda & : & E & \to & \R^+ \\
 & &  (x,y) & \mapsto & \lambda x + \frac{\pi^2}{2 y^2}, \\
\end{array}
\eeq
and for any $\mu$ in $M_p(E)$,
\begin{equation}
\label{eq:psi}
\Psi^\lambda (\mu):=\inf_{(x,y)\in \mu} \psi^\lambda(x,y).
\end{equation}

We can now define the variable that appears as the limit law in Theorem \ref{thm0},
\beq
F := \Psi^{\gl(\gb)}(\Pi) = \inf_{(x,y) \in \Pi} \Big\{\gl(\gb)x + \frac{\pi^2}{2y^2}\Big\},
\eeq
where we remind that $(x,y) \in \Pi$ means that $(x,y)$ is in the support of $\Pi$.
We shall write $F^\beta$ instead of $F$ when we want to stress the dependence on this parameter.

\par The rest of the section is devoted to various results relative to the environment $\Pi$.
\begin{proposition}
\label{prop:unique}
For all $\lambda>0$, $\Psi^{\lambda}(\Pi)$ is almost surely well defined and positive. Moreover the infimum in the definition of $\Psi^{\lambda}$ is almost surely achieved at a unique point $(x^*,y^*)$ in $E$.
\end{proposition}
\begin{proof}[Proof of Proposition~\ref{prop:unique}]
\par As there is almost surely at least one point in $\Pi$, the infimum is well defined. Moreover, there are almost surely finitely many points in $[0,1]\times [1,+\8)$, which implies that $\bar Y$, the maximum of the second coordinates among these points, is almost surely finite. Thus, $\psi^\lambda(x,y)\geq \min\{{\pi^2}/{(2\bar Y^2)}, \lambda\}$ for any $(x,y)\in \Pi$, which proves that $\Psi^{\lambda}(\Pi)$ is almost surely positive.
\par For $\lambda >0$ and $u>0$, we introduce the set
\beq
\label{eq:defA}
A^\lambda_{u}:=\{(x,y)\in E \textrm{ such that }\psi^{\lambda}(x,y)\leq u\},
\eeq
and note that $\Pi (A^\lambda_{u})<+\8$ almost surely.
\par Let us denote by $C$ the event that the minimum in the definition of $\Psi$ is achieved. We observe that for any $u>0$, 
\beq
P(\Psi(\Pi)<u)=P(\Psi(\Pi)<u, \Pi(A_u^\lambda)<+\8)\leq P(\Psi(\Pi)<u, \Pi \in C).
\eeq
As $\lim_{u\to +\8}P(\Psi(\Pi)<u)=1$ and $\lim_{u \to +\8}P(\Psi(\Pi)<u, \Pi \in C)=P(\Pi \in C)$, we obtain that $P(\Pi \in C)=1$.
\par It remains to prove that the infimum is almost surely achieved at only one point.
For $u>0$, let $D_{u}$ be the event that two points of $\Pi \cap A_u^\lambda$ have the same image by $\psi^{\lambda}$ and $D$ be the event that two points of $\Pi$ have the same image by $\psi^{\lambda}$.
For $n\in\N$ and conditional on $\{\Pi(A_u^\lambda)=n\}$, the restriction of $\Pi$ to $A_u^\lambda$ has the same law as $\sum_{i=1}^n \delta_{X_i}$, where the $(X_i)_{1\le i\leq n}$'s are i.i.d. with continuous law $p$ restricted to $A_u^\lambda$ and renormalised to a probability measure. This implies that $P(\Pi\in D_u | \Pi(A_u^\lambda))=0$ and thus $P(\Pi\in D_u)=0$. We obtain
\beq
P(\Psi(\Pi)<u) = P(\Psi(\Pi)<u, \Pi \in D_u^c).
\eeq
The first term converges to $1$ while the second one converges to $P(\Pi \in D)$ when $u$ goes to infinity. This proves $P(\Pi \in D)=1$ and concludes the proof. 
\end{proof}

\par Unfortunately, it is difficult to work directly with $\Psi^\lambda$ as it appears not to be a vaguely continuous function on $M_p(E)$. For this reason, we introduce the function $\Psi^{\lambda}_K$, for any compact set $K\subset E$, defined by
\begin{equation*}
\Psi^\lambda_K (\mu):=\inf_{(x,y)\in \mu \cap K} \psi^\lambda(x,y),\qquad \mu\in M_p(E).
\end{equation*}
The advantage of restricting the infimum to a compact set lies in the following lemma. 
\begin{lemma}\label{prop:continuity}
For any compact set $K\subset E$ and $\lambda>0$, the  function $\Psi^\lambda_K$ is vaguely continuous on $M_p(E)$.
\end{lemma}
\begin{proof}
Consider a sequence $(\mu_n)_{n\geq 1}$ in $M_p(E)$ that converges vaguely to $\mu$. Due to Proposition~$3.14$ in \cite{Res}, $\mu \in M_p(E)$. Suppose $\mu(K)=k$, that is $\mu(\cdot \cap K)$ writes $\sum_{i=1}^k \delta _{(x_i,y_i)}$ where $(x_i,y_i)_{1 \leq i \leq k}$ is a family in $K$. By Proposition~$3.13$ in \cite{Res} there exists for all $n$ larger than some $n(K)$ a family $(x^n_i,y^n_i)_{1 \leq i \leq k}$ such that $\mu_n(\cdot \cap K)=\sum_{i=1}^k \delta _{(x^n_i,y^n_i)}$. Moreover, for all $1\leq i \leq k$, the sequence
$(x^n_i,y^n_i)_{n\geq 1}$ converges to $(x_i,y_i)$ as $n$ goes to infinity. This implies that $\Psi^{\lambda}_K(\mu_n)$ converges to  $\Psi^{\lambda}_K(\mu)$.
\end{proof}

\par We conclude this section with the following technical lemma: 
\begin{lemma}
\label{lem:Fgep}
The family $(F^{\beta-\epsilon})_{\gep \geq 0}$ (seen as functions of the random measure $\Pi$) converges non-decreasingly to $F^{\beta}$ when $\epsilon\to 0$, almost surely.
\end{lemma}
\begin{proof}
Using Proposition \ref{prop:unique} we can define almost surely a random point $(X^*(\beta),Y^*(\beta))$ such that
\beq
\psi^{\lambda(\beta)}(X^*(\beta),Y^*(\beta))=\Psi^{\lambda(\beta)}(\Pi).
\eeq
We first prove that $(X^*(\beta-\epsilon),Y^*(\beta-\epsilon))= (X^*(\beta),Y^*(\beta))$ almost surely if $\epsilon>0$ is chosen small enough. Let $\gep_0\in(0,\gb)$.
Consider some $x\geq x_0:=\frac{2\Psi^{\lambda(\beta)}(\Pi)}{\lambda(\beta-\epsilon_0)}$ and any $y>0$. As $\lambda(\cdot)$ and  $\Psi^{\lambda(\cdot)}$ are non-decreasing functions of $\beta$ it holds that 
\beq
\psi^{\lambda(\beta-\epsilon)}(x,y) \geq \lambda(\beta-\epsilon) x \geq 2\Psi^{\lambda(\beta-\epsilon)}(\Pi)
\eeq
for all $\epsilon<\epsilon_0$, and we conclude that $X^*(\beta-\epsilon)<x_0$ for any $\epsilon<\epsilon_0$.

\par Consider now some $y\leq y_0:= \frac{\pi}{2\sqrt{\Psi^{\lambda(\beta)}(\Pi)}}$ and any $x>0$. 
We thus obtain 
\beq
\psi^{\lambda(\beta-\epsilon)}(x,y) \geq \frac{\pi^2}{2 y^2} \geq  2\Psi^{\lambda(\beta-\epsilon)}(\Pi),
\eeq
and we conclude that $Y^*(\beta-\epsilon)>y_0$ for any $\epsilon<\epsilon_0$.
We deduce thereof that almost surely
\beq 
F^{\beta-\epsilon}= \inf_{(x,y)\in \Pi \cap \{x<x_0 , y>y_0\}}\psi^{\lambda(\beta-\epsilon)}(x,y).
\eeq

Finally, observe that $\Pi (x<x_0 , y>y_0) <+\8$ almost surely, so that there are only finitely many candidates for $(X^*(\beta-\epsilon),Y^*(\beta-\epsilon))$. Pick $(X,Y)\in   \Pi \cap \{x<x_0 , y>y_0\}$ that is not $ (X^*(\beta),Y^*(\beta))$ (if there is no such point there is nothing more to prove as $(X^*(\beta),Y^*(\beta))$ is then the only candidate). The function $\epsilon \mapsto \lambda(\beta-\epsilon) X + \frac{\pi^2}{2Y^2}$ is (i) strictly larger than $\lambda(\beta) X^*(\beta) + \frac{\pi^2}{2(Y^*(\beta))^2}$ at $\epsilon=0$ due to Proposition \ref{prop:unique} and (ii) continuous due to Proposition \ref{lem:continuity_lambda}. Therefore,
\beq
\lambda(\beta-\gep) X + \frac{\pi^2}{2Y^2}>\lambda(\beta) X^*(\beta) + \frac{\pi^2}{2(Y^*(\beta))^2}
\eeq
for $\epsilon>0$ small enough. As $\Pi( \{x<x_0 , y>y_0\}) <+\8$ almost surely, we can choose $\epsilon$ small enough so that the last inequality holds for all points in $\Pi \cap \{x<x_0 , y>y_0\}$. Therefore, we may from now on consider $\epsilon>0$ small enough so that $(X^*(\beta-\epsilon),Y^*(\beta-\epsilon))= (X^*(\beta),Y^*(\beta))$.

\par From what precedes, we get
\beq
|\Psi^{\lambda(\beta-\epsilon)}(\Pi)-\Psi^{\lambda(\beta)}(\Pi)|\leq |\lambda(\beta-\epsilon)-\lambda(\beta)| X^*(\beta).
\eeq
Again, as $\lambda$ is continuous, $\Psi^{\lambda(\beta-\epsilon)}(\Pi)$ converges almost surely to $\Psi^{\lambda(\beta)}(\Pi)$ when $\epsilon$ goes to $0$. Finally,
\beq
 F^{\beta-\epsilon}  \underset{\epsilon \to 0}{\to}  F^{\beta} \qquad a.s.
\eeq
Moreover, as $\lambda$ is non-decreasing with $\beta$, the convergence is monotone:
\beq
 F^{\beta-\epsilon}  \underset{\epsilon \to 0}{\nearrow}  F^{\beta} \qquad a.s.
\eeq
\end{proof}

\subsection{Statement of the result}
\label{subsec:sor}
From now on we set
\beq
\label{eq:norm}
N = N(n) = n^{\frac{\gamma}{\gamma+2}},\qquad F_n = - \frac{1}{N} \log Z_n,\qquad n\ge1.
\eeq

\par Let us first explain at a heuristic level the choice of this normalization. The argument is of the type one uses to find volume exponents in some polymer models and is sometimes referred to as a {\it Flory argument}. We assume that at a large time $n$ the walk has visited at most $N$ traps and has remained confined in the largest visible gap, and we find the value of $N$ with the best energy-entropy balance. By basic extreme-value theory, the size of that gap is of order $N^{1/\gamma}$, and by a standard small-ball estimate (see Proposition~\ref{pr:small_ball} below for a precise statement) the entropic cost of being confined in that gap during time $n$ is of order $nN^{-2/\gamma}$. Also, the cost of crossing $N$ traps should be roughly of order $N$, see Proposition~\ref{pr:lambda.beta} below for a rigorous proposition. Finally, by equating these two costs, one finds the optimal choice $N= n^{\frac{\gamma}{\gamma +2}}$. As a consequence, the walk has travelled a distance of order $n^{\frac{1\vee \gamma}{\gamma+2}}$ from the origin (see Proposition \ref{thm:limit.stable}). This {\it confinement strategy} will be justified during the proof of our theorem, which we now state:
\begin{theorem}\label{thm0}
The sequence of random variables $(F_n)_{n\ge 1}$ converges in $\bbP$-distribution to the random variable
\beq
\label{eq:th}
F := \Psi^{\gl(\gb)}(\Pi) = \inf_{(x,y) \in \Pi} \Big\{\gl(\gb)x + \frac{\pi^2}{2y^2}\Big\},
\eeq
where $\Pi$ is a Poisson point process on $\bbR^+\times\bbR^+_*$ with intensity $p =\dd x \otimes c_{\tau} \gamma \  y^{-(1+\gamma)}\dd y$. 
\end{theorem}

\par The rest of the article is devoted to the proof of this result. The general strategy is the following. In Section \ref{sec:KT} we introduce the notion of \textit{good environments}, that are environments for which we can provide suitable bounds on the survival probability $Z_n$, see~\eqref{eq:ge}. In Sections~\ref{sec:LB} and \ref{sec:UB} we successively prove lower and upper bounds, assuming only that the environment is good. We complete the proof in Section \ref{sec:pot}: first we show that we can adjust parameters so that environments are asymptotically good, with a probability arbitrarily close to one;  then we use the two bounds obtained in the previous sections to relate the quantity of interest to an explicit functional $\Psi$ of the properly renormalized environment, see \eqref{eq:psi}. Finally we let $n$ tend to infinity to prove that $F$ has the same law as $\Psi$ applied to the limiting environment, that is a Poisson point process (see Proposition \ref{prop:convergenceP}).

\subsection{Comments} 
\label{subsec:c}
We collect here a number of comments and remarks about our result.\\

{\noindent \bf 1.} We do not consider the annealed survival probability, which decreases at most polynomially fast since 
\beq
\bbE\bP(\gs \wedge H_{\Z^-}>n) \ge \frac 12 \bP_1(H_0 \geq n) \bbP(\tau_1 > n) \sim \frac 12  c_\tau\ n^{-\gga-\frac 12},
\eeq
as $n\to \infty$, and thus has a completely different behaviour than the quenched probability.\\

{\noindent \bf 2.} Note that we cannot hope for better than weak convergence. Indeed, if $F_n$ would converge  to $F$ almost-surely, then $F$ would be measurable with respect to the tail $\gs$-algebra of the family $(T_i)_{i\ge 1}$. As the increments are independent, the latter is trivial due to the $0-1$ law, and $F$ would be deterministic.\\

{\noindent \bf 3.} In the case $\gga\le 1$, the variational formula in \eqref{eq:th} admits an alternative representation in terms of a subordinator, which reads
\beq
F = \inf_{t\ge 0} \Big\{\gl(\gb)t + \frac{\pi^2}{2(\gD\cS_t)^2}\Big\},
\eeq
where $(\cS_t)_{t\ge 0}$ is a $\gga$-stable subordinator and $\gD\cS_t = \cS_t - \cS_t^- = \cS_t - \lim_{u\to t^-} \cS_u$.\\

{\noindent \bf 4.} We can compute explicitly the tail distribution function of the limiting law $F = \Psi^{\gl(\gb)}(\Pi)$ in Theorem \ref{thm0}. Recall \eqref{eq:def_p}, \eqref{eq:psi0} and \eqref{eq:defA}. For any $u \ge 0$ (we write $\gl$ instead of $\gl(\gb)$ to lighten notations),
\beq
\ba
\bbP(F \ge u)&=\bbP(\Pi(A^\lambda_{u})=0)=\exp(- p(A^\lambda_{u})).
\ea
\eeq
Since $A^\gl_u = \Big\{(x,y)\colon 0 \le x < u/\gl,\ y\ge \frac{\pi}{\sqrt{ 2(u-\gl x)}} \Big\}$, we get by a straightforward computation that
\beq
\label{eq:Fcont}
\bbP(F \ge u) = \exp\Big( - \frac{c_\tau}{\gl(\gb)\pi^\gamma (\gamma+2)} (2u)^{\frac{\gga}{2}+1} \Big),\qquad u\ge 0.
\eeq

{\noindent \bf 5.} The case $\gamma=0$ is left open. In this case, a gap distribution of the form \eqref{eq:tail.ass} is no longer appropriate and one should instead assume that $\bbP(T_1=n) \sim L(n)/n$, where $L$ is a non-negative slowly varying function such that $\sum L(n)/n$ is finite. Complications may arise at two levels~: (i) the normalization of $\max_{1\le i\le n} T_i$, that we use to guess the value of $N$, and (ii) the integrability of $\log T_1$, that we use in Proposition~\ref{pr:lambda.beta}. For instance, if $L(n) = (\log n)^{-2}$ then $\bbE(\log T_1) = \infty$ and $\max_{1\le i\le n} T_i$ has a completely different type of renormalization since, as one can readily show, $(1/\sqrt{n}) \log \max_{1\le i\le n} T_i$ converges to a non-trivial probability law with cumulative distribution function $x \mapsto \exp(-x^{-2})\ind_{\{x>0\}}$, as $n\to\8$.\\

{\noindent \bf 6.} We state without proof an alternative expression for $\gl(\gb)$ based on ergodic theory considerations. To this end, let $\tilde\tau$ be an independent copy of $\tau$, as defined in Section~\ref{sec:traps}. Suppose that the random walk is now free to visit $\Z^-$ but is killed by the set $-\tilde\tau$ (note the minus sign), with the same probability $1-\exp(-\gb)$, and denote by $\tilde\gs$ the corresponding killing time. Then,
\beq
\label{eq:alt_lambda.beta}
\gl(\gb) = - \bbE \tilde\bbE \log \bP^\gb(H_{\tau_1} < \tilde\gs).
\eeq
Assuming this last equality, we could readily prove using the dominated convergence theorem that $\lambda$ is also continuous at $0$.\\

{\noindent \bf 7.} Equation~\eqref{eq:encadr.lambda} does not give much information about the behaviour of $\lambda(\gb)$ at $0$, that remains an open question. We expect however that $\gb = o(\gl(\gb))$ as $\gb\to 0$ and we now explain why. To this end, recall~\eqref{eq:alt_lambda.beta} and the related notations above. By integrating over $\cN$ and differentiating in $\beta$ we obtain
\beq
\lim_{\beta \to 0} \gl'(\beta) = \tilde\bbE\bbE\bE\Big(\sum_{k=1}^{H_{\tau_1}} \ind_{\{S_k \in-\tilde\tau\}} \Big),
\eeq
that we expect to be infinite. Indeed, by first restricting the walk to make its first step to the left and then using the symmetry of the random walk,
\beq
\tilde\bbE\bbE\bE\Big(\sum_{k=1}^{H_{\tau_1}} \ind_{\{S_k \in-\tilde\tau\}} \Big)
\ge \frac12 \tilde\bbE\bE_{-1}\Big(\sum_{k=1}^{H_0} \ind_{\{S_k \in-\tilde\tau\}} \Big)
= \frac12 \tilde\bbE\bE_1\Big(\sum_{k=1}^{H_0} \ind_{\{S_k \in \tilde\tau\}} \Big).
\eeq
We now interchange integrals and use the renewal theorem to obtain, at least for $\gamma \neq 1$,
\beq
\lim_{\beta \to 0} \gl'(\beta) \ge \frac12 \bE_1\Big(\sum_{k=1}^{H_0} \tilde\bbP (S_k \in \tilde\tau) \Big) 
\ge \frac{C}{2} \bE_1\Big(\sum_{k=1}^{H_0} (1+S_k)^{(\gga-1)\wedge 0} \Big).
\eeq
Since, by Ray-Knight's theorem, the mean number of visits to $x\in\bbN_0$ between time $1$ and $H_0$ equals $1$ under $\bP_1$, we get
\beq
\lim_{\beta \to 0} \gl'(\beta) \ge C \sum_{x\ge 0} (1+x)^{(\gga-1)\wedge 0} = \infty.
\eeq

{\noindent \bf 8.} Note that we offer no path statement. In other words, we do not prove anything about the behaviour of the walk conditioned to survive for a long time $n$. However, as it is often the case with this type of model, our result and the method of proof suggest a path strategy, which in our context corresponds to a confinement (or localization) strategy. To be more precise, we roughly expect that as $n$ is large, the walk reaches the trap labelled $X^*(\beta)N$ and then remains in the corresponding gap, of size $Y^*(\beta)N^{1/\gga}$, where $(X^*(\beta),Y^*(\beta))$ is distributed as the unique minimizer of the random variational problem in Theorem~\ref{thm0}. In other words, the path (or polymer) gets stuck in a slab which is {\it large enough} and {\it not too far from the origin}. Surprisingly, the repulsive interaction between the path and the traps leads to a pinning effect on the quenched path measures, as explained by Sznitman~\cite[Chapter 6]{Sz98}. Proving such a result should demand substantial additional work, as one would most likely need sharper bounds on the survival probability (partition function) and fine controls on the ratio of survival probabilities restricted to suitable events. Nevertheless, this can be considered as an interesting direction of research. Let us mention that Ding and Xu~\cite{DingXu} recently obtained a path confinement result in the case of quenched hard Bernoulli obstacles for $d\ge 2$.\\

{\noindent \bf 9.} Let us stress that the scaling $t^{\gamma/(\gamma+2)}$ that appears in our Theorem \ref{thm0} is different from the scaling of the PAM in a bounded i.i.d. potential. In this case \cite[Example $5.10$]{Ko16} states that the correct scaling is $t$ up to a logarithmic correction. Hence we are in a case where the correlations of the potential have a drastic effect on the asymptotic behaviour of the survival probability.

\section{Key tools}
\label{sec:KT}
In this section we introduce several tools which we will use to prove our result. For convenience, various notations are gathered together in Section~\ref{subsec:not}, to which the read may refer. In Section~\ref{subsec:conf} we remind the reader of the so-called small ball estimate and establish a rough upper bound on the probability that the walker stays confined in a fixed gap until time $n$, see Proposition~\ref{prop:roughUB}. Section~\ref{subsec:di} contains Proposition \ref{prop:FKG}, which states that a walk conditioned to hit a given point $x$ before coming back to $0$ does it faster when it is also conditioned on survival until its first visit to $x$. In Section~\ref{subsec:ttl} we state the two technical Lemmas~\ref{lem:control_ratio} and~\ref{lem:control_pinbad} that we will use in Section~\ref{sec:UB} while proving the upper bound on $Z_n$. Finally we introduce the key notion of \textit{good environment} in Section~\ref{subsec:ge}. Informally, \textit{good environments} are those for which we are able to efficiently bound $Z_n$. We thus give a list of events, see \eqref{eq:ge}, that are essentially the technical conditions we will need in Proposition~\ref{prop:ub} and~\ref{prop:lb}.

\subsection{Notations} \label{subsec:not} Let us introduce notations that will be necessary in what comes next.\\

{\it \noindent Killing clock.} We recall the definition of the $\N$-valued geometric random variable $\cN$ with success parameter $1-e^{-\beta}$ that plays the role of the killing clock.\\

{\it \noindent Records.}  As we already hinted in the heuristics, only the {\it largest} gaps matter. To be more precise, a particular attention is given to record values of the sequence $(T_\ell)$. Therefore, we let
\beq
\label{def:ik}
i(0)=0,\qquad i(k) = \inf\{i>i(k-1)\colon T_{i+1} > T_{i(k-1)+1}\},\qquad k\ge 1,
\eeq
be the sequence of record indexes, while
\beq
\label{eq:def.t.tau.star}
\tau^*_k=\tau_{i(k)} \qquad \textrm{ and } \qquad  T^*_k = T_{i(k)+1},\qquad k\ge 0.
\eeq
We also define
\beq
R(a,b) = \{k \ge 1 \colon a\le i(k) \le b\},\qquad \cR(a,b) = i(R(a,b)),\qquad a,b\in\bbN,\ a<b,
\eeq
and
\beq
\label{eq:def.R.cR}			
R_\gep(n) = R(\gep N, \gep^{-1}N),\qquad \cR_\gep(n) = \cR(\gep N, \gep^{-1}N), \qquad n\in \bbN, \qquad \gep>0.
\eeq
Finally we write 
\beq
\cR=\cR(1,+\8),
\eeq
for the set of all records.\\

{\it \noindent Auxiliary random walk.} We remind that the clock process $(\gt_n)_{n\ge 1}$ is defined by
\beq
\gt_0 = 0,\quad \gt_{n+1} = \inf\{k > \gt_n \colon S_k \in \tau\},\quad n\ge 0.
\eeq
The process that chronologically keeps track of the traps visited by the walk will be denoted by $X = (X_n)_{n\ge 0}$ and is uniquely determined by $\tau_{X_n} = S_{\gt_n}$. 
It is not difficult to see that $X$ is a Markov chain on $\bbN_0$, the hitting times of which are denoted by
\beq
\label{def:zeta}
\gz_x = \inf\{ n\ge 1 \colon X_n = x\},\qquad x\in\bbN_0,
\eeq
and
\beq
\label{def:zeta.star}
\gz^*_k = \inf\{ n\ge 1 \colon X_n = i(k)\},\qquad x\in\bbN_0.
\eeq

\medskip
{\it \noindent Transition matrices and their moment-generating functions.} Let us define
\beq
\label{eq:defqij}
q_{ij}(n) = \bP_{\tau_i}(S_k \notin \tau,\ 1\le k < n,\ S_n = \tau_j),\qquad i,j\in \bbN_0, \quad n\ge 1,
\eeq
and the associated family of matrices $\{Q(\phi)\}_{\phi\ge 0}$ by
\beq
\label{eq:defQij}
Q_{ij}(\phi) = \sum_{n\ge 1} e^{\phi n} q_{ij}(n) = \bE_{\tau_i}\left(e^{\phi \theta_{1}}\ind_{\{S_{\theta_1}=\tau_j\}}\right),\qquad i,j\in \bbN_0, \ \phi \geq 0.
\eeq
Note that the matrix $\{Q_{ij}(0)\}_{i,j\ge 0}$ is nothing but the transition matrix of the Markov chain $X$ defined above. These quantities will appear in Lemma~\ref{lem:control_ratio} below and are zero as soon as $|j-i|>1$. Finally, we will also use the following notations for the gap associated to an non-oriented edge $\{i,j\}$ with $|j-i|\le 1$:
\beq
\label{eq:deftij}
t_{ij} = 
\left\{
\begin{array}{lll}
t_{i+1} & \text{if} & j=i+1, \\
t_i & \text{if} & j = i- 1, \\
t_{i+1} \vee t_i & \text{if} & i=j,
\end{array}
\right.
\eeq
where $(t_i)$ is a sequence of integers.
\medskip
\subsection{Confinement estimates} \label{subsec:conf}
One of the key standard estimates in our arguments are the so-called {\it small-ball estimates}, that control the probability that a simple random walk stays confined in an interval:
\begin{proposition} \label{pr:small_ball}
There exist $t_0,c_1,c_2,c_3,c_4>0$ such that for all $t>t_0$, the following inequalities hold for all $n\ge 1$ such that $n\in 2\bbN$ or $n-t \in 2\bbN$:
\beq
\label{eq:sb}
\frac{c_1}{t \wedge n^{1/2}} e^{-g(t)n} \le \bP(H_t \wedge H_0 \wedge H_{-t} > n) \le \frac{c_2}{t \wedge n^{1/2}} e^{-g(t)n}
\eeq
\beq
\frac{c_3}{t^3 \wedge n^{3/2}} e^{-g(t)n} \le \bP(H_t \wedge H_0 \wedge H_{-t} = n) \le \frac{c_4}{t^3 \wedge n^{3/2}} e^{-g(t)n},
\eeq
where
\beq
\label{eq:g}
g(t) = -\log \cos \big(\frac{\pi}{t}\big) = \frac{\pi^2}{2t^2} + O\big(\frac{1}{t^4}\big),\qquad t\to +\8.
\eeq
\end{proposition}

This proposition is taken from Lemma 2.1 in Caravenna and Pétrélis~\cite{CP09b}. A crucial point here is the uniformity of the constants, which gives the uniformity of the constant $C$ in Proposition~\ref{prop:roughUB}.

\par Caravenna and Pétrélis~\cite{CP09b} treated the case of equally spaced traps, which we refer to as the {\it homogeneous} case, in the sense that increments of $\tau$ are all equal. We summarize their results here.

\begin{proposition}[Homogeneous case, see Eq. (2.1)-(2.3) in \cite{CP09b}]
\label{pr:homo}
Let $t\in \bbN$ and $\tau = t\bbZ$. There exists a constant $\phi(\gb, t)$ such that 
\beq 
\label{eq:315}
\phi(\gb, t) = - \lim_{n\to\infty} \frac{1}{n} \log \bP(\gs > n),
\eeq
with 
\beq
\label{eq:phi}
\phi(\gb, t) = \frac{\pi^2}{2t^2} \Big(1 - \frac{4}{e^{\gb} - 1}\frac{1}{t} + o\Big(\frac{1}{t} \Big) \Big).
\eeq
Moreover, it is the only solution of the equation:
\beq
\bE(\exp(\phi \inf\{n\ge 1 \colon S_n \in \tau\}))) = \exp(\gb),\qquad \gb\ge 0.
\eeq
\end{proposition}
Note that the first order term in the expansion of $\phi$ does not depend on $\gb$.
It turns out that we are able to extend this proposition, at the price of additional technical work, to deal with the {\it weakly-inhomogeneous} case, that is when increments of $\tau$ follow a periodic pattern. We obtain the following:

\begin{proposition}[Weakly-inhomogeneous case] 
\label{pr:periodic}
Let $p\ge 2$, $t_1, \ldots, t_p$ positive integers and $\tau$ be the periodic set $\{\tau_i \colon 0\le i < p\} + \tau_p \bbZ$, where $\tau_0 = 0$ and $\tau_i = \sum_{1\le j \le i} t_j$ for all $0<i<p$.
There exists a constant $\phi = \phi(\gb ; t_1, \ldots, t_p)$ such that 
\beq 
\phi(\gb ; t_1, \ldots, t_p) = - \lim_{n\to\infty} \frac{1}{n} \log \bP(\gs > n).
\eeq
Moreover,
\beq
\bP(\gs > n) \le Cn^2p \exp(-\phi(\gb  ; t_1, \ldots, t_p) n),\quad n\geq 1,
\eeq
and
\beq \label{eq:comp.phi}
 \phi(\gb, t_{\max}) \le \phi(\gb ; t_1, \ldots, t_p) < g(t_{\max}), \qquad t_{\max} = \max_{1\le i \le p} t_i.
\eeq
\end{proposition}

The proof is deferred to Appendix~\ref{sec:periodic}. Remark that both inequalities in \eqref{eq:comp.phi} are intuitive: the first one asserts that it is easier to survive in a homogeneous environment with gap $t_{max}$ than in the original environment. The second one states that one of the strategy to survive is to stay confined in the largest gap. With this estimate in hand, we get our first building block, that is an upper bound on the probability to survive in-between two traps, for a general environment $\tau$.

\begin{proposition}\label{prop:roughUB}
There exists a constant $C>0$ such that for all $0\le k < r <  \ell$, one has
\beq
\bP_{\tau_r}(\sigma \wedge H_{\tau_k} \wedge H_{\tau_\ell} > n) \le Cn^2(\ell-k) \exp(-\phi(\gb;\max\{t_i \colon k<i\le \ell\})n),
\eeq
where $\phi(\gb;\cdot)$ is defined in Proposition~\ref{pr:homo}.

\end{proposition}

\begin{proof}[Proof of Proposition~\ref{prop:roughUB}]
The proof relies on periodization. Since the random walk does not leave the interval $(\tau_k,\tau_\ell)$ on the event considered, we may as well replace the renewal $\tau$ by a periodized version, and by translation invariance, consider that the random walk starts at zero. Finally, everything is as if the walk evolves in a new environment $\tilde\tau$, with periodic increments, defined by
\beq
\tilde{\tau}=\{\tau_i \colon k\le i \le \ell \} -\tau_r+ (\tau_\ell-\tau_k) \bbZ,
\eeq
and we now have to bound from above 
$
\bP^{\tilde\tau}(\sigma \wedge H_{\tilde\tau_{k-r}} \wedge H_{\tilde\tau_{\ell-r}} > n),
$
where we put a superscript on $\bP$ to stress that the walk evolves among $\tilde\tau$. This probability is certainly smaller than $\bP_0^{\tilde\tau}(\sigma> n)$, and we may now conclude thanks to Proposition~\ref{pr:periodic}.
\end{proof}

%%%%%%%%%%%%%%%%%%%%%%%%%%%%%%%%

\subsection{A decoupling inequality} \label{subsec:di} %
The next building block is a control on the probability that the walk reaches a given point before a certain fixed time, conditional on survival and not coming back to $0$. 
In essence, the following proposition tells us that the walk reaches this point {\it stochastically faster} in the presence of traps:

\begin{proposition}\label{prop:FKG}
For all $\gb>0$, $x\in \bbN$ and $n\in \N$,
\beq
\label{eq:propFKG}
\bP^{\gb}(H_x \le n\ |\ \gs\wedge H_0 > H_x) \ge \bP(H_x \le n\ |\ H_0 > H_x).
\eeq
\end{proposition}
Let us stress that this proposition is general, as it does not depend on the position of the traps.

\begin{proof}[Proof of Proposition~\ref{prop:FKG}]
Let $x\in\bbN$. We first remark that the stopped process $(S_{k\wedge H_x})_{k\ge 0}$ is still a Markov chain under $\bP^{\gb}(\cdot | \gs\wedge H_0 > H_x)$, with $0$ as initial state, $x$ as absorbing state, and transition probabilities given by
\beq
\ba
\bar Q_\gb(a,b) := 
\begin{cases}
\frac{e^{-\gb\ind_{\{b\in\tau\}}}\bP_b(\gs\wedge\tilde H_0 > \tilde H_x)}{e^{-\gb\ind_{\{a+1\in\tau\}}}\bP_{a+1}(\gs\wedge\tilde H_0 > \tilde H_x) + e^{-\gb\ind_{\{a-1\in\tau\}}}\bP_{a-1}(\gs\wedge\tilde H_0 > \tilde H_x)} & \text{if } |b-a| = 1\\
0 & \text{otherwise,}
\end{cases}
\ea
\eeq
where $1\le a < x$ and $\tilde H_z := \inf\{n\ge 0\colon S_n = z\}$. By applying the strong Markov property at $H_{a+1}$, observe that
\beq
\frac{\bar Q_\gb(a,a-1)}{\bar Q_\gb(a,a+1)} = \frac{e^{-\gb\ind_{\{a-1\in\tau\}}}\bP_{a-1}(\gs\wedge\tilde H_0 > \tilde H_x)}{e^{-\gb\ind_{\{a+1\in\tau\}}}\bP_{a+1}(\gs\wedge\tilde H_0 > \tilde H_x)}
= e^{-\gb\ind_{\{a-1\in\tau\}}}\bP_{a-1}(\gs\wedge\tilde H_0 > H_{a+1}),
\eeq
and note that this ratio is non-increasing in $\gb$, for all $1\le a < x$. We may deduce by a standard coupling argument that $H_x$ is stochastically smaller under $\bP^{\gb}(\cdot | \gs\wedge H_0 > H_x)$ than under $\bP(\cdot | H_0 > H_x)$, which corresponds to the case $\gb=0$. This concludes the proof.
\end{proof}
 
\subsection{Two technical lemmas} \label{subsec:ttl}
\par Recall the notations in~\eqref{eq:defQij} and~\eqref{eq:deftij}.
\begin{lemma}
\label{lem:control_ratio}
Define the function $f \colon z\in (0,\pi) \mapsto z/\sin(z)$. Let $n\in\bbN$. For $\gep>0$ small enough, there exist $\ga=\ga(\gep)>0$, $C>0$ and $T_0(\gep)\in\bbN$ such that for $T> T_0(\gep) \vee \max_{0\le x \le n} T_x$ and $\phi = \frac{\pi^2}{2T^2}$,
\beq
\label{eq:ratio}
\frac{Q_{x,y}(\phi)}{Q_{x,y}(0)}
\leq
\begin{cases}
\exp(\gep) & \text{if } x\neq y  \text{ and } T_{xy} \le \alpha T, \text{ or } x=y\\
2 f\left(\pi \frac{\max_{0\leq x\leq n}T_{x}}{T}(1+\frac{C}{T^2})\right) & \text{else.}
\end{cases}
\eeq
\end{lemma}
The ratio in \eqref{eq:ratio} is the positive Laplace transform of the hitting time of $\tau$ for the walk conditioned to go from $\tau_x$ to $\tau_y$. As we will consider large values of $T$ and thus small values of $\phi$, Lemma \ref{lem:control_ratio} can be understood as a control of this transform near $0$.
\begin{proof}[Proof of Lemma~\ref{lem:control_ratio}]
\par We consider $\gep>0$ small enough (it will appear in the following how small it has to be). Let us start with the case $x\neq y$. From the explicit expressions of the Laplace transforms, see e.g. (A.5) in Caravenna and Pétrélis~\cite{CP09}, we get
\beq
\label{eq:defQxy}
Q_{x,y}(\phi) = \frac{\tan \gD}{2\sin(T_{xy}\gD)}, \quad \text{where } \gD=\gD(\phi) = \arctan(\sqrt{e^{2\phi} - 1}),
\eeq
and we note that
\beq
\label{eq:defQxy0}
Q_{x,y}(0) = \frac{1}{2T_{xy}}.
\eeq
\par Let us notice that \eqref{eq:defQxy} is well-defined if $T_{xy} \gD < \pi$, which occurs as soon as $T$ is large enough. Indeed, by expanding $\gD$ we see that there exists a constant $C>0$ and $T_1\in\bbN$ such that for $T\ge T_1$,
\beq
\label{eq:UBgD}
\gD \le \frac{\pi}{T}\Big(1+\frac{C}{T^2}\Big).
\eeq
If we assume moreover that $T>T_{xy}$, we obtain, as $T$ and $T_{xy}$ are integers,
\beq
\label{eq:UBTxygD}
\frac{T_{xy}\gD}{\pi} \le \frac{T-1}{T}\Big(1+\frac{C}{T^2}\Big) = 1- \frac{1+o(1)}{T}<\frac{1}{1+\gep},
\eeq
provided $T$ is larger than some $T_1(\gep) \in \bbN$. For the rest of the proof, we assume that $T> T_1 \vee T_1(\gep) \vee \max\{T_x\colon 0\le x\le n\}$.
\par By combining \eqref{eq:defQxy} and \eqref{eq:defQxy0}, we now obtain
\beq
\frac{Q_{x,y}(\phi)}{Q_{x,y}(0)} = \frac{T_{xy}\tan (\gD)}{\sin(T_{xy}\gD)}.
\eeq
By using \eqref{eq:UBgD} and expanding $\tan$ to first order, there exists $T_2(\gep)$ such that for $T\ge T_2(\gep)$,
\beq
\tan(\gD) \le (1+\gep) \gD.
\eeq
By adding this latter condition on $T$, we get, since $f$ is increasing,
\beq
\begin{aligned}
\label{eq:UBratioQ's}
\frac{Q_{x,y}(\phi)}{Q_{x,y}(0)} &\le (1+\gep) \frac{T_{xy}\gD}{\sin(T_{xy}\gD)} \\
&= (1+\gep)f(T_{xy}\gD)\\
&\le (1+\gep)f\left(\pi \frac{T_{xy}}{T}\Big(1+\frac{C}{T^2}\Big)\right).
\end{aligned}
\eeq
As $\gep<1$, that concludes the proof of the second inequality in our statement.
To get the first inequality when $x\neq y$, notice first that, as $f(z) \to 1$ when $z \to 0$, there exists $z_\gep$ such that $(1+\gep)f(z)\le \exp(\gep)$ for $z\le z_\gep$. Therefore, it is enough to define 
\beq
\ga(\gep) = \frac{z_\gep}{\pi(1+\gep)},
\eeq
assume that $T> T_3(\gep):=(C/\gep)^{1/2}$ and use \eqref{eq:UBratioQ's} to get what we need.\\
We are left with the case $x = y$. Again, with the help of (A.5) in Caravenna and Pétrélis~\cite{CP09},
\beq
Q_{xx}(\phi) = 1 - \frac12 \frac{\tan(\gD)}{\tan(T_{x-1,x}\gD)} - \frac12 \frac{\tan(\gD)}{\tan(T_{x,x+1}\gD)},
\eeq
where $\gD$ is defined as in \eqref{eq:defQxy}. We thereby retrieve the standard formula:
\beq
Q_{xx}(0) = 1 - \frac{1}{2T_{x-1,x}} - \frac{1}{2T_{x,x+1}}.
\eeq
Note that it is enough to treat the case $T_{x,x+1} = T_{x-1,x}$ since
\beq
Q_{xx}(\phi) = \frac12\Big(1 - \frac{\tan(\gD)}{\tan(T_{x-1,x}\gD)}\Big) + \frac12\Big(1 - \frac{\tan(\gD)}{\tan(T_{x,x+1}\gD)}\Big).
\eeq
We may now consider the ratio
\beq
\label{eq:Qratio}
\frac{Q_{xx}(\phi)}{Q_{xx}(0)}
=
\frac{1 - \frac{\tan(\gD)}{\tan(T_{x,x+1}\gD)}}{1- \frac{1}{T_{x,x+1}}}.
\eeq
By choosing $T\ge T_2(\gep)$ and expanding $\tan$ to first order, we obtain
\beq
1 - \frac{\tan(\gD)}{\tan(T_{x,x+1}\gD)} \le 
\begin{cases}
1 - \frac{\gD}{\tan(T_{x,x+1}\gD)} & \text{if } \quad T_{x,x+1}\gD \le \frac{\pi}{2},\\ 
1 - (1+\gep) \frac{\gD}{\tan(T_{x,x+1}\gD)} & \text{if } \quad  \frac{\pi}{2} <T_{x,x+1}\gD < \pi.
\end{cases}
\eeq
We remind that our conditions on $T$ guarantee that $T_{x,x+1}\gD < \pi$. The reason why we split cases above is that $\tan$ changes sign at the value $\pi$. We further make a dichotomy : (i) $T$ is large and $T_{x+1}$ is at least $\gep T$ and (ii) $T$ is large and $T_{x+1}$ less than $\gep T$. Let us start with (i).  If actually $T_{x,x+1}\gD < \pi/2$, we may simply bound the denominator in \eqref{eq:Qratio} by $1$. Otherwise, we note that $z\mapsto - z/\tan(z)$ is increasing on $(\pi/2,\pi)$, so we may write, as $T_{x,x+1}\gD < \pi/(1+\gep)$ by \eqref{eq:UBTxygD},
\beq
1 - (1+\gep)\frac{\gD}{\tan(T_{x,x+1}\gD)} \le 1 - \frac{1}{T_{x,x+1}} \frac{\pi}{\tan(\pi/(1+\gep))}.
\eeq
Thus, if we define
\beq
T_4(\gep) = \frac{3}{\gep^2}\Big(\frac{\pi}{|\tan(\pi/(1+\gep))|} \vee 1 \Big)
\eeq
and assume that $T>T_4(\gep)$, we obtain
\beq
\frac{Q_{xx}(\phi)}{Q_{xx}(0)} \le \frac{1-\frac{1}{T_{x,x+1}}\frac{\pi}{\tan(\pi/(1+\gep))}}{1-\frac{1}{T_{x,x+1}}} \le \frac{1+\gep/3}{1-\gep/3} = 1+\frac 23\gep+o(\gep), \quad \gep \to 0,
\eeq
which is smaller than $\exp(\gep)$ when $\gep$ is small enough.
We now deal with (ii) and to this end we assume $T\ge T_2(\gep)$ and $T_{x,x+1} \le \gep T$, in which case we expand $\tan(T_{x,x+1}\gD)$. {By using \eqref{eq:UBgD} and assuming that $T>T_3(\gep)$ we get $T_{x,x+1} \gD \le \gep(1+\gep)\pi$}. Thus, there exists a constant $C=C(\gep)>0$ such that
\beq
\tan(T_{x,x+1}\gD) \le T_{x,x+1}\gD + C(T_{x,x+1}\gD)^3, \quad \text{so } 1 - \frac{\gD}{\tan(T_{x,x+1}\gD)} \le 1 - \frac{1}{T_{x,x+1}}(1-CT_{x,x+1}^2\gD^2).
\eeq
Finally, since $T_{x,x+1}\ge 2$ necessarily,
\beq
\frac{Q_{xx}(\phi)}{Q_{xx}(0)} \leq 1 + 2CT_{x,x+1}\gD^2 \le 1 +2C\gep(1+\gep)\gD.
\eeq
Now we pick $T\ge T_5(\gep)$ such that $\gD \le [2C(1+\gep)]^{-1}$ and we get the claim.
\par We conclude the proof by setting $T_0(\gep) = \max(T_1, T_1(\gep),T_2(\gep),T_3(\gep), T_4(\gep),T_5(\gep))$.
\end{proof}
Recall the notations in~\eqref{eq:def.t.tau.star} and~\eqref{eq:deftij}. Given $\alpha>0$ and $k\in\bbN$, we define a set of {\it bad} edges as
\beq
\label{eq:defB}
\cB_{k,\alpha} = \{1\le x, y\le i(k) \colon x\neq y,\ T_{x,y} > \alpha T^*_{k}\}.
\eeq
and its cardinal
\beq
L_{k,\alpha} = |\cB_{k,\alpha}|.
\eeq
These bad edges correspond to the second case in Lemma \ref{lem:control_ratio}.
Recall also~\eqref{def:zeta} and~\eqref{def:zeta.star}. The following lemma controls the visits to the bad edges:

\begin{lemma}
\label{lem:control_pinbad}
There exists  a function $h$ such that, for any $A>0$, $k\geq 0$, and $\alpha>0$, if $T^*_k > h(A,L_{k,\alpha},\alpha)$ then
\beq 
\bE^{\gb}\Big( A^{\sharp\{i\le \gz^*_k \colon \{X_{i-1},X_i\} \in \cB_{k,\alpha}\}} \ind_{\{\gz^*_k< \gz_0 \wedge \cN\}} \Big) \leq 2A^{L_{k,\alpha}} \tau^*_k  \bP^{\gb}(\gz^*_k< \gz_0\wedge \cN).
\eeq
\end{lemma}
\begin{proof}[Proof of Lemma~\ref{lem:control_pinbad}]
\par We start with the case $L_{k,\alpha}=1$ and denote by $(s,s+1)$ the bad edge. By using the geometric nature of $\cN$ and applying the Markov property at $\gz_{s+1}$, we get
\beq
\ba
\label{eq:cpinbad1}
\bE^{\gb}\Big( & A^{\sharp\{i\le \gz_k^* \colon \{X_{i-1},X_i\} \in \cB_{k,\alpha}\}}\ind_{\{\gz_k^* < \gz_0 \wedge \cN\}} \Big)=\bE^{\gb}\Big( A^{\sharp\{i\le  \gz_{s+1} \colon \{X_{i-1},X_i\} = \{s,s+1\}\}}\ind_{\{\gz_k^* < \gz_0 \wedge \cN\}} \Big)\\
&\le \bP^\gb(\gz_{s+1} < \gz_0 \wedge \cN)A\ 
\bE_{\tau_{s+1}}^{\gb}\Big( A^{\sharp\{i\le  \gz_{s+1} \colon \{X_{i-1},X_i\} = \{s,s+1\}\}} \ind_{\{\gz_k^* < \gz_0 \wedge \cN\}} \Big),
\ea
\eeq
and we now focus on the last factor in the line above. By considering the consecutive visits of $X$ to $s+1$, we may write
\beq
\bE_{\tau_{s+1}}^{\gb}\Big( A^{\sharp\{i\le  \gz_{s+1} \colon \{X_{i-1},X_i\} = \{s,s+1\}\}} \ind_{\{\gz_k^* < \gz_0 \wedge \cN\}} \Big)
= \bE(v^G) \bP_{\tau_{s+1}}(\gz_k^* < \cN | \gz_k^* < \gz_{s+1}),
\eeq
where $G$ is a $\N_0$-valued geometric random variable with parameter $\bP_{\tau_{s+1}}(\gz_k^* < \gz_{s+1})$ and 
\beq
v = \bE_{\tau_{s+1}}\Big( A^{\sharp\{i\le  \gz_{s+1} \colon \{X_{i-1},X_i\} = \{s,s+1\}\}} \ind_{\{ \gz_{s+1} < \cN\}} |  \gz_{s+1} < \gz_k^* \Big).
\eeq
We are going to show now that $v \le 1$ when $T^*_k \geq h_0(A,\alpha)$ where 
\beq
\label{eq:defH}
h_0(A,\alpha)=\frac{A^2}{2\alpha e^{\beta}(e^{\beta}-1)}.
\eeq
To this end, note that
\beq
v \le \frac12 e^{-\gb} + \frac12 \Big(1 - \frac{1}{T_{s,s+1}} \Big)e^{-\gb} + \frac{1}{2T_{s,s+1}}A^2 e^{-2\gb}.
\eeq 
Indeed, the first term is given by walks which make their first step to the right. The second term comes from those who make their first step to the left but come back to $\tau_{s+1}$ before hitting $\tau_s$, whereas the third term comes from the walks who hit $\tau_s$ before coming back to $\tau_{s+1}$.  
Then, as $T_{s,s+1}\geq \alpha T^*_k$,
\beq
v \leq e^{-\beta}+\frac{A^2e^{-2\beta}}{2\alpha T^*_k},
\eeq
which, by \eqref{eq:defH}, proves that $v\leq 1$.
To complete the proof in this case, we write
\beq
\ba
\text{r.h.s\eqref{eq:cpinbad1}} &\le A\ \bP^\gb(\gz_{s+1} < \gz_0 \wedge \cN)\bP_{\tau_{s+1}}(\gz_k^* < \cN | \gz_k^* < \gz_{s+1})\\
&\le A\  \bP^\gb(\gz_{s+1} < \gz_0 \wedge \cN)
\frac{\bP^\gb_{\tau_{s+1}}(\gz_k^* < \cN\wedge \gz_{s+1})}{\bP_{\tau_{s+1}}(\gz_k^* < \gz_{s+1})}\\
&\le 2A(\tau_k^* - \tau_{s+1}) \bP^\gb(\gz_{s+1} < \gz_0 \wedge \cN)\bP^\gb_{\tau_{s+1}}(\gz_k^* < \cN\wedge \gz_{s+1})\\
&\le 2A \ \tau_k^*\  {\bP^\gb(\gz_k^* < \gz_0 \wedge \cN)}.
\ea
\eeq
\par Let us now conclude the proof in the general case $L_{k,\alpha}\ge 1$. Our strategy is to decouple the contribution of each bad set by Holdër's inequality and reduce the problem to the case $L_{k,\alpha}=1$ with $A$ replaced by $A^L$. Indeed, if we note
\beq
\cB_{k,\alpha}= \{(s_\ell, s_{\ell+1}) \colon 1\le \ell \le L,\ 1\le s_\ell < i(k) \},
\eeq
and suppose that $T_k^*\geq h(A,L_{k,\alpha},\alpha)$, where
\beq
\label{eq:def.gALga}
h(A,L,\alpha):=h_0(A^L,\alpha)= \frac{A^{2L}}{2\alpha e^{\beta}(e^{\beta}-1)},
\eeq
we get
\beq
\ba
&\bE^{\gb}\Big( A^{\sharp\{i\le \gz_k^* \colon \{X_{i-1},X_i\} \in \cB_{k,\alpha}\}}\ind_{\{\gz_k^* < \gz_0 \wedge \cN\}} \Big)\\
&\le
\prod_{i=1}^{L_{k,\alpha}}
\bE^{\gb}\Big( (A^{L_{k,\alpha}})^{\sharp\{i\le \gz_k^* \colon \{X_{i-1},X_i\} = \{s_\ell, s_{\ell+1}\}\}}\ind_{\{\gz_k^* < \gz_0 \wedge \cN\}} \Big)^{1/L_{k,\alpha}}\\
&\le 2A^{L_{k,\alpha}} \tau_k^*  {\bP^\gb(\gz_k^* < \gz_0 \wedge \cN)}.
\ea
\eeq
This concludes the proof.
\end{proof}

\subsection{Good environments}
\label{subsec:ge}
{We define here a notion of \textit{good environments}, that are environments where it is possible to give a good control on the survival probability. We will show in Section~\ref{sec:proba.good.env} that these environments are typical, meaning that by tuning some parameters and considering $n$ large enough, their probability of occurence can be made arbitrarily close to one.}

\subsubsection{Additional notations} Beforehand, we remind of the functions $f$ and $h$ introduced in Lemma \ref{lem:control_ratio} and Lemma \ref{lem:control_pinbad}. We define 
\beq
\label{eq:def_fk}
f_{k}:=2 f\left(\pi \frac{T^*_{k-1}}{T^*_k}\Big[1+\frac{C}{(T^*_k)^2}\Big]\right),
\eeq
that appears in the events $A_n^{(6)}$ and $A_n^{(7)}$ below. The constant $C$ above is the same as the one in \eqref{eq:ratio} in Lemma \ref{lem:control_ratio}.
From \eqref{eq:g} and \eqref{eq:phi}, there exists (for a fixed $\gb$) a constant $C_1>0$ so that 
\beq
\label{eq:encadr_g}
1/(C_1 t^2) \leq g(t) \wedge \phi(\gb,t) \leq g(t) \vee \phi(\gb,t) \le C_1 / t^2, \qquad t\ge 1. 
\eeq
This constant appears in the event $A_n^{(9)}$. Finally, we define the exponent
\beq
\label{def:kappa}
\gk =
\begin{cases}
\frac{\gamma}{4} & \text{if } \gamma \le 1\\
\frac{1}{2\gamma} - \frac14 & \text{if } 1 <\gamma < 2\\
\frac{1}{2\gamma} & \text{if } \gamma \ge 2,
\end{cases}
\eeq
which appears in the event $A_n^{(1)}$.

\subsubsection{Definition} Let $\delta,\gep_0,\gep,\eta>0$. The set of good environments, denoted by $\Omega_n(\gd,\gep_0,\gep,\eta)$, is defined as the intersection of the events defined below (we stress that $\alpha(\gep)$ and $T_0(\gep)$ appearing here are the same as in Lemma \ref{lem:control_ratio}):
\beq
\Omega_n(\gd,\gep_0,\gep,\eta)=\bigcap_{i=1}^{11} A_n^{(i)}(\gd,\gep_0,\gep,\eta),
\eeq
with
\beq
\label{eq:ge}
\ba
A^{(1)}_n&=
\begin{cases}
{\{\tau_{N^{1+\gk}}^2 < n^{1-\frac{\gamma\wedge (2-\gamma)}{4(\gamma+2)}}\}} & \text{if } \gga < 2\\
{\{\tau_{N^{1+\gk}}^2 < n^{1+\frac{2\gga - 1}{2(\gamma+2)}}\}} & \text{if } \gga \ge 2,\\
 \end{cases}\\
A_n^{(2)}(\gep_0) &:= {\{ T_k \le \gep_0^{\frac{1}{2\gamma}} N^{\frac{1}{\gamma}},\quad \forall k\le \gep_0N\}}\\
A_n^{(3)}(\gep_0) &:= \{\tau_{N/\gep_0} < n\}\\
A_n^{(4)}(\gd,\gep) &:=\{\exists \ell \in \{N,\ldots, 2N\}\colon T_\ell \ge T_0 \vee \gd N^{\frac{1}{\gamma}}\}\\
A_n^{(5)}(\gep_0,\gep) &:=\{\forall k\in R_{\gep_0}(n),\ T^*_k > T_0 \vee \gep_0^{\frac{3}{2\gamma}}N^{\frac{1}{\gamma}}\}\\
A_n^{(6)}(\gep_0,\gep) &:=\{\forall k\in R_{\gep_0}(n),\ f_k^{L_k} \le \exp(n^{\frac{\gamma}{2(\gamma+2)}})\}\\
A_n^{(7)}(\gep_0,\gep) &:=\{\forall k\in R_{\gep_0}(n),\ T^*_k > h(f_k, L_{k,\ga}, \ga)\}\\
A_n^{(8)}(\gep_0) &:= \{|R(1,N/\gep_0)| \le [\log(N/\gep_0)]^2\}\\
A_n^{(9)}(\gd) &:=\{ |\gl(2N,\gb) - \gl(\gb)| \le \tfrac{C_1}{2\gd^{2}}\}\\
A_n^{(10)}(\gep_0,\gep,\eta) &:= \{ |\gl(\ell-1,b) - \gl(b)| \le \tfrac{\gep_0\eta}{2},\ \forall \ell\ge \gep_0N,\ b\in \{\beta, \beta-\gep\}\}\\
A^{(11)}_n(\gep_0)&:=\{ \Pi_N( \{y>1/\gep_0\}) = 0 \}.
\ea
\eeq
We might omit some or all parameters of $\Omega_n$ when no confusion is possible to alight notations. 
Event $A_n^{(1)}$ is used to provide a lower bound on $Z_n$, see Proposition \ref{prop:lb}. Events from $A_n^{(2)}$ to $A_n^{(9)}$
 are used to establish the upper bound, see Proposition \ref{prop:ub}. Finally, the events $A_n^{(10)}$ and $A_n^{(11)}$ are used in the conclusion of the proof, see \eqref{eq:ubGtilde} and \eqref{eq:useAn10} in Section \ref{subsec:conclusion}.
 
\section{Lower bound}\label{sec:LB}
In this section we prove a lower bound on $Z_n$ that is an upper bound on $F_n$.\\
Let us set for $\gb \ge 0$, $n\in\bbN$ and $\ell > 1$,
\beq
\label{eq:defG}
G^\beta_n(\ell) = -\frac{1}{N}\log \bP^\beta(\gs \wedge H_0> H_{\tau_{\ell-1}}) + g(T_\ell)\frac{n}{N} = \gl(\ell-1,\gb)\frac{\ell-1}{N}+ g(T_\ell)\frac{n}{N},
\eeq
where $\gl(\ell,\gb)$ has been defined in~\eqref{def:gl.ell.gb}. Recall the definition of $\gk$ in~\eqref{def:kappa}.
Our goal in this section is to prove
\begin{proposition}
\label{prop:lb}
On $\Omega_n$,
\beq
F_n \le  \min_{1< \ell \le N^{1+\gk}} G_n(\ell) + o(1) 
\eeq
where the $o(1)$ holds when $n$ goes to $+\8$.
\end{proposition}
Actually, only $A_n^{(1)}$ is necessary in Lemma \ref{lem:lem.proplb} and this event does not depend on any parameter. That is why we omit the parameters in $\gO_n$ above. 

In order to prove this proposition we need the following lemma, that states that the exponential cost of reaching level $N^{1+\gk}$ (and so any level $\ell \leq N^{1+\gk}$) before time $n$ is negligible in front of $N$.
\begin{lemma}
\label{lem:lem.proplb}
There exists a function $k(N)$ such that $k(N)=o(N)$ as $N\to\8$ and, on $\gO_n$, for all $1< \ell <N^{1+\kappa}$,
\beq
\bP(H_{\tau_{\ell}}\le n | H_{\tau_{\ell}} < H_0 ) \ge \exp(-k(N)).
\eeq
\end{lemma}

\begin{proof}[Proof of Lemma \ref{lem:lem.proplb}]
Observe that 
\beq
\ba
\bP(H_{\tau_{\ell}}\le n | H_{\tau_{\ell}} < H_0 ) &\geq \bP(S_n \geq \tau_{\ell}\  ;\  S_k >0,\  0<k\leq n )\\
&\geq \bP(S_n \geq \tau_{\ell})\bP( S_k >0,\  0<k\leq n )\\
&\geq \bP(S_n \geq \tau_{N^{1+\kappa}})\bP( S_k >0,\  0<k\leq n ).
\ea
\eeq
To go from the first to the second line we use the FKG inequality, since both events are non-decreasing coordinate-wise in the family of i.i.d. increments $(S_i-S_{i-1})_{1\leq i \leq n}$. As there exists some constant $C>0$ such that $\bP( S_k >0,\  0<k\leq n )\geq C/\sqrt{n}$ we only focus on $\bP(S_n \geq \tau_{N^{1+\kappa}})$.

Assume first that $\gamma<2$. Then, we notice that, on $A^{(1)}_n$, $\tau_{N^{1+\gk}}= o(\sqrt{n})$ when $n$ goes to infinity so that 
\beq
\bP(S_n \geq \tau_{N^{1+\kappa}})=\bP(S_n \geq 0)-\bP(0\leq S_n < \tau_{N^{1+\kappa}})=1/2+o(1).
\eeq
Assume now that $\gamma\ge 2$. In this case, $\tau_{N^{1+\gk}}$ is not anymore negligible in front of $\sqrt{n}$. However, on $A_n^{(1)}$, a standard large deviation estimate for the simple random walk yields
\beq
\bP(S_n \geq \tau_{N^{1+\kappa}})\ge  \exp\Big(-Cn^{\frac{2\gga-1}{2(\gga+2)}}\Big).
\eeq
We may conclude by setting $k(N) = CN^{1-\frac{1}{2\gga}}$.
\end{proof}

\begin{proof}[Proof of Proposition \ref{prop:lb}]
We provide a lower bound on $Z_n$ by computing the cost of various strategies. Here is an informal description of the tested strategies: for  $1 < \ell \le N^{1+\gk}$, the walk reaches $\tau_{\ell-1}$ before time $n$ and before its death, which has a probability of order $e^{-\lambda \ell}$. Then, it remains confined in the gap $(\tau_{\ell-1}, \tau_{\ell})$ until time $n$, with a probability of order $e^{- g(T_\ell) n}$. We finally optimise on $1 < \ell \le N^{1+\gk}$.
\par We may thus write for all $\ell > 1$,
\beq
Z_n \ge Z_n^{[\ell]} := \bP(H_{\tau_\ell}\wedge \gs \wedge H_{0} > n \ge H_{\tau_{\ell-1}}),
\eeq
and then optimize { over $1 < \ell \le N^{1+\gk}$}.
By decomposing on the value of $H_{\tau_{\ell-1}}$, we have on $\Omega_n$ and for $n$ large enough,
\beq
\label{eq:lb}
\ba
Z_n^{[\ell]} \ &{ \geq } \sum_{0\le k \le n} \bP(\gs \wedge H_{0} > H_{\tau_{\ell-1}} = k) \bP_{\tau_{\ell-1}}(H_{\tau_\ell} \wedge H_{\tau_{\ell-1}} > n-k)\\
& \ge \sum_{0\le k \le n} \bP(\gs  \wedge H_{0} > H_{\tau_{\ell-1}} = k) \bP_{\tau_{\ell-1}}(H_{\tau_\ell} \wedge H_{\tau_{\ell-1}} > n)\\
& \ge  \bP(\gs \wedge H_{0}> H_{\tau_{\ell-1}},  H_{\tau_{\ell-1}}\le n) \times \frac{c_1}{2\sqrt{n}} e^{- g(T_\ell) n} \\
& \ge  \bP(\gs \wedge H_{0}> H_{\tau_{\ell-1}})\bP(H_{\tau_{\ell-1}}\le n | H_{0}>H_{\tau_{l-1}}) \times \frac{c_1}{2\sqrt{n}} e^{- g(T_\ell) n}\\
& \ge \frac{c_1}{2\sqrt{n}} \bP(\gs  \wedge H_{0} > H_{\tau_{\ell-1}})e^{- g(T_\ell) n - k(N)}.
\ea
\eeq
Note that we have used Proposition~\ref{pr:small_ball} to go from the second to the third line and  Proposition~\ref{prop:FKG} to go from the third to the fourth line. Finally, to go from the fourth to the fifth line we use Lemma~\ref{lem:lem.proplb} and the fact that the environment is in $\Omega_n$.
Therefore,
\beq
F_n \le -\frac1N \log \Big( \frac{c_1}{2\sqrt{n}} \Big) + \frac{k(N)}{N} + \inf_{ 1< \ell \le N^{1+\gk}} G_n^\gb(\ell),
\eeq
where $G_n^\gb$ is defined in \eqref{eq:defG}. Since
\beq
-\frac1N \log \Big( \frac{c_1}{2\sqrt{n}} \Big) + \frac{k(N)}{N} = o(1),
\eeq
as $N\to \8$, this concludes the proof.
\end{proof}

\section{Upper bound}\label{sec:UB}
In this section we prove an upper bound on $Z_n$ or, equivalently, a lower bound on $F_n$. Recall the definitions in \eqref{eq:def.R.cR}.
\begin{proposition}
\label{prop:ub}
Let $\epsilon,\delta>0$. There exists $\gep_0>0$ such that, on $\Omega_n(\gd,\gep_0,\gep)$,
\beq
\label{eq:goal}
F_n \geq \min_{\ell\in \mathcal{R}_{\gep_0}(n)} G_n^{\beta-\epsilon}(\ell)+o(1).
\eeq
where the $o(1)$ holds as $n$ goes to infinity and $G_n^\gb$ is defined in \eqref{eq:defG}.
\end{proposition}
Before starting the proof of this proposition we need additional notations. Recall~\eqref{def:ik}. We define the hitting time of the $k$-th record
\beq
H^*_k = H_{\tau_{i(k)}}, \qquad k\ge 0.
\eeq
We also define 
\beq
\tilde H_0 = 0,\qquad \tilde H_i = \inf\{n> \gt_{i-1} \colon S_n \in \tau\} - \tilde H_{i-1}, \qquad i\ge 1.
\eeq
For all $n\geq 1$ and $k\geq 1$, we define
\beq
Z_n^{(k)} = Z_n(H^*_k \le n < H^*_{k+1})
\eeq
where $Z_n(A)= \bP^{\gb}(\sigma \wedge H_{\Z^-} > n, A)$ for any event $A$. The quantity above corresponds to the survival probability restricted to the event that the walker has reached the $k$-th record but not the $(k+1)$-th one. These events are clearly incompatible for different values of $k$.
\begin{proof}[Proof of Proposition \ref{prop:ub}] Let $\epsilon,\delta>0$.
We choose $\gep_0$ small enough so that 
\beq
\label{eq:gep0}
\ba
\frac{\beta}{\gep_0}& >2(C_1 \delta^{-2} + \gl(\gb)),\\
 \gep_0^{-1/\gamma}& >4C_1(C_1 \delta^{-2} + \gl(\gb)),\\
\ea
\eeq
(these technical conditions will become clear soon).

We have to prove that \eqref{eq:goal} is satisfied on $\Omega_n$. We thus consider until the end of this proof an environment that is in $\Omega_n$.

\par \textbf{Step $1$.} We first prove that  we can actually consider only the records in $R_{\gep_0}(n)$, that are the ones lying in the window $\{\gep_0 N , \cdots, N/\gep_0\}$, see \eqref{eq:step1} for a precise formulation. As the environment is in $A_n^{(1)} \cap A_n^{(4)}(\gd,\gep)$, using \eqref{eq:lb} and \eqref{eq:encadr_g}, we obtain for $n$ large enough, a rough lower bound on $Z_n$:

\beq
\ba
Z_n \ &{\geq } \max_{ N \leq \ell \leq 2N} Z_n^{[\ell]}  \\
& \ge \max_{N \leq \ell \leq 2N} \frac{c_1}{2\sqrt{n}} \bP(\gs \wedge H_0 > H_{\tau_{\ell-1}})e^{- g(T_\ell) n - k(N)}\\
& \ge  \frac{c_1}{2\sqrt{n}} \exp\Big(- \Big\{ C_1\delta^{-2}N + k(N)+ 2N \gl(2N,\gb)\Big\}\Big).
\ea
\eeq

As the environment is in $A_n^{(9)}(\gd)$, we finally get
\beq
\label{1}
Z_n \ \geq  \frac{c_1}{2\sqrt{n}} \exp\Big\{-N \Big(2C_1 \delta^{-2} + 2\gl(\gb) + o(1)\Big)\Big\},
\eeq
where the $o(1)$ holds as $n\to\8$.
Observe that 
\beq
\label{2}
 \sum_{k\in R(N/\gep_0,+\8)}  Z_n^{(k)}\leq e^{-\beta N/\gep_0},
\eeq
while due to Proposition \ref{prop:roughUB} and the fact that the environment is in $A_n^{(2)}(\gep_0)$, we have for $n$ large enough,
\beq
\ba
\label{3}
 \sum_{k\in R(0,\gep_0 N)}  Z_n^{(k)}\leq Z_n(H_{\tau_{\gep_0 N}}>n)&\le  Cn^2 (\gep_0 N) \exp(-\phi(\gb, \max\{T_i,\ i\le \gep_0 N\})n)\\
 & \le Cn^2 (\gep_0 N) \exp(-\frac{n}{C_1(\gep_0^{1/2\gamma} N^{1/\gamma})^2})\\
 & \le \exp \Big({-\frac{ \gep_0^{-1/\gamma}N}{2C_1}}\Big).
\ea
\eeq

Combining \eqref{1},  \eqref{2} and  \eqref{3} and due to the choice of $\gep_0$ in \eqref{eq:gep0}, we easily get that for $n$ large enough

\beq
\label{eq:step1}
Z_n \leq 2 \sum_{k \in R_{\gep_0}(n)}  Z_n^{(k)}.
\eeq

\textbf{Step $2$.} The main work is thus to derive an upper bound on $Z_n^{(k)}$ for $k \in  R_{\gep_0}(n)$ (see \eqref{eq:step2}).

\par Using the Markov property at time $H^*_k$ we observe that for $k \in  R_{\gep_0}(n)$
\beq
\ba
\label{eq:auxZnk}
Z_n^{(k)}&=\bP^\beta \left( \sigma \wedge H_{\Z^-}> n, H^*_k \le n < H^*_{k+1} \right)\\
	&=\sum_{m=0}^n \bP^\beta \left( \sigma \wedge H_{\Z^-} > m, H^*_k =m \right)\bP_{\tau_{i(k)}}^\beta \left( \sigma \wedge H^*_{k+1}  \wedge H_{\Z^-}>n-m \right).
\ea
\eeq
For all $m\geq 0$, by Proposition \ref{prop:roughUB}, we have on $A^{(3)}_n(\gep_0)$, for $n$ large enough,
\beq
\ba
\label{eq:auxZnk2}
\bP_{\tau_{i(k)}}^\beta \left( \sigma \wedge H^*_{k+1} \wedge H_{\Z^-} >n-m \right)&\leq C (n-m)^2 i(k+1)e^{-\phi(\beta;\max\{T_i\colon 0\leq i< i(k+1)\})(n-m)}\\
			&\leq C\ n^3\ e^{-\phi(\beta, T^*_k)(n-m)}.
\ea			
\eeq

\par It remains to bound $\bP^\beta \left( \sigma \wedge H_{\Z^-} > m, H^*_k =m \right)$ for $0\le m \le n$.
Recall~\eqref{def:zeta} and~\eqref{def:zeta.star}.
By using Tchebychev's inequality in the first line and then conditioning on $X$, we obtain for 
$\phi:=\frac{\pi^2}{2(T^*_k)^2}$,
\beq
\ba
\label{eq:auxZnk3}
\bP^\beta \left( \sigma \wedge H_{\Z^-} > m, H^*_k =m \right)&=\bP^\gb(\gz^*_{k} < \gz_0 \wedge \cN,\ H^*_{k} = m) \\
&\leq e^{-\phi m} \bE^\gb(e^{\phi H^*_{k}} \ind_{\{\gz^*_{k} < \gz_0 \wedge \cN\}})\\
&\leq e^{-\phi m} \bE^\gb( \bE(e^{\phi H^*_{k}}| X) \ind_{\{\gz^*_{k} < \gz_0 \wedge \cN\}})\\
&\leq e^{-\phi m} \bE^\gb\Big( \prod_{1\le i \le \gz^*_{k}} \bE(e^{\phi \tilde H_i}| X_{i-1}, X_i) \ind_{\{\gz^*_{k} < \gz_0 \wedge \cN\}} \Big).\\
\ea
\eeq
Next, by integrating on $\cN$ we obtain
\beq
\label{eq:UB0}
\bE^\gb\Big( \prod_{1\le i \le \gz^*_{k}} \bE(e^{\phi \tilde H_i}| X_{i-1}, X_i) \ind_{\{\gz^*_{k}< \gz_0 \wedge \cN\}} \Big)=\bE\Big( \prod_{1\le i \le \gz^*_{k}} e^{-\gb} \frac{Q_{X_{i-1},X_i}(\phi)}{Q_{X_{i-1},X_i}(0)} \ind_{\{\gz^*_{k} < \gz_0\}} \Big)
\eeq
with notations similar to \eqref{eq:defQij}.

\par On $A_n^{(5)}(\gep_0,\gep)$ the assumptions of Lemma~\ref{lem:control_ratio} are valid (with $T_k^*$ playing the role of $T$), which provides $\ga>0$. Recall the definition of $\cB_{k,\ga}$ in \eqref{eq:defB}. We obtain
\beq
\ba
\label{eq:interle}
&\bE\Big( \prod_{1\le i \le \gz^*_{k}} e^{-\gb} \frac{Q_{X_{i-1},X_i}(\phi)}{Q_{X_{i-1},X_i}(0)} \ind_{\{\gz^*_{k} < \gz_0\}} \Big)\\
		 &\qquad\leq \bE^{\gb-\gep}\Big( 2 f\left(\pi \frac{T^*_{k-1}}{T^*_k}(1+\frac{C}{(T^*_k)^2})\right)^{\sharp\{i\le \gz_k^* \colon \{X_{i-1},X_i\} \in \cB_{k,\ga}\}} \ind_{\{\gz^*_k < \gz_0 \wedge \cN\}} \Big).
\ea
\eeq
Recall the definition of $f_k$ in~\eqref{eq:def_fk}. On $A^{(7)}_n(\gep_0,\gep)$ the assumptions of Lemma \ref{lem:control_pinbad} (with $f_{k}$ playing the role of $A$) are satisfied and from \eqref{eq:interle} we obtain that on $A^{(3)}_n (\gep_0)$
\beq
\ba
\label{eq:auxZnk4}
\bE\Big( \prod_{1\le i \le \gz^*_k} e^{-\gb} \frac{Q_{X_{i-1},X_i}(\phi)}{Q_{X_{i-1},X_i}(0)} \ind_{\{\gz^*_k< \gz_0\}} \Big)
		& \leq \bE^{\gb-\gep}\Big( f_{k}^{\sharp\{i\le \gz_k^* \colon \{X_{i-1},X_i\} \in \cB_{k,\alpha}\}} \ind_{\{\gz^*_k < \gz_0 \wedge \cN\}} \Big)\\
		& \leq 2 f_{k}^{L_{k,\alpha}} n\  \bP^{\gb-\epsilon}(\gz^*_k < \gz_0\wedge \cN).
\ea
\eeq

\par Finally, combining \eqref{eq:auxZnk}, \eqref{eq:auxZnk2}, \eqref{eq:auxZnk3} and \eqref{eq:auxZnk4} we obtain for $k\in R_{\gep_0}(n)$
\beq
\label{eq:inter}
\ba
Z_n^{(k)} &\leq 2C\ n^4\  f_{k}^{L_{k,\alpha}} \ e^{-\phi(\beta, T^*_k)n}\  \bP^{\gb-\epsilon}(\gz^*_k < \gz_0\wedge \cN) \sum_{m=0}^n e^{-\Big(\frac{\pi^2}{2(T^*_k)^2} - \phi(\gb,T^*_k) \Big)m}\\
 & \leq 2C\ n^4 f_{k}^{L_{k,\alpha}} e^{-G_n^{\beta-\epsilon}(i(k))N}   e^{-\left(\phi(\beta ; T^*_k )-g(T^*_k)\right)n} \sum_{m=0}^n e^{-\Big(\frac{\pi^2}{2(T^*_k)^2} - \phi(\gb,T^*_k) \Big)m}.
\ea
\eeq
On $A_n^{(5)}(\gep_0,\gep)$ we control both the errors $\phi(\beta ; T^*_k )-g(T^*_k)$ and $\phi(\gb,T^*_k)-\frac{\pi^2}{2(T^*_k)^2}$.
Indeed due to \eqref{eq:g} and \eqref{eq:phi} there exists some constant $C(\beta)$ such that for $t$ large enough (depending on $\beta$), 
\beq
\label{eq:erreur}
|\phi(\beta ; t)-\tfrac{\pi^2}{2t^2}| \vee |\phi(\beta ; t)-g(t)|<\frac{C(\beta)}{t^3},
\eeq
and we obtain that on $A_n^{(5)}(\gep_0,\gep)$ and for $n$ large enough
\beq
|\phi(\beta ; T^*_k)-\tfrac{\pi^2}{2(T^*_k)^2}|  \vee |\phi(\beta ; T^*_k )-g(T^*_k)|\leq { \frac{C(\beta)}{\gep_0^{9/2\gamma}N^{3/\gamma}}.}
\eeq
From \eqref{eq:inter} we thus obtain for $n$ large enough and $k\in R_{\gep_0}(n)$
\beq
Z_n^{(k)} \leq 2C\ n^5  f_{k}^{L_{k,\alpha}} \exp\Big\{-N \min_{\ell \in \mathcal{R}_{\gep_0}(n)}  G_n^{\beta-\epsilon} (\ell)+\  C(\beta)\gep_0^{-9/2\gamma}n^{\frac{\gamma-1}{\gamma+2}} \Big\}.
\eeq
We also remind that on $A^{(6)}_n(\gep_0,\gep)$, $f_{k}^{L_{k,\alpha}} \leq e^{n^{\frac{\gamma}{2(\gamma+2)}}}$ so that finally for $k \in R_{\epsilon_0}(n)$, and for $n$ large enough,
\beq
\label{eq:step2}
Z_n^{(k)} \leq 2C\ n^5\exp\Big(-N \min_{\ell \in \mathcal{R}_{\gep_0}(n)}  G_n^{\beta-\epsilon} (\ell) +o(N)\Big).
\eeq
\textbf{Step $3$.} It remains to sum this last equation for $k \in R_{\gep_0}(n) $. As the environment is in $A_n^{(8)}(\gep_0)$, we finally obtain that for $n$ large enough
\beq
Z_n \leq \gep_0^{-1} 2C\ n^6\exp\Big(-N \min_{\ell \in \mathcal{R}_{\gep_0}(n)}  G_n^{\beta-\epsilon}(\ell)  +o(N)\Big)
\eeq
and 
\beq
F_n \geq \min_{\ell \in \mathcal{R}_{\gep_0}(n)} G_n^{\beta-\epsilon}(\ell) +o(1), 
\eeq
where the $o(1)$ holds as $n\to\8$.
\end{proof}

\section{Proof of Theorem \ref{thm0}} \label{sec:pot}
This section is divided in two parts. In the first one we establish that, for $n$ large enough, environments are good with probability arbitrary close from $1$. The second one is devoted to the proof our main result Theorem \ref{thm0}. Due to the control on the environment we can actually restrict our analysis to the event of good environments so that results of Proposition \ref{prop:lb} and \ref{prop:ub} are in force and provide a precise control on $Z_n$.

\subsection{Environments are good with high probability}
The aim of this part is to prove the following proposition that assures a control of the probability that an environment is good:

\begin{proposition}\label{lem:ge}
For all $\gt>0$ there exists $\delta$ and $\gep_1(\delta)$ small enough such that for all $\gep_0<\gep_1(\delta)$, for all $\gep,\eta>0$
\beq
\liminf_{n\to +\8} \bbP\left(\Omega_n(\delta,\gep_0,\gep,\eta) \right) \ge 1-\gt.
\eeq
\end{proposition}

We first establish various results on the \textit{records} of the sequence $(T_i)_{i\geq 1}$.
\subsubsection{Preliminaries on records}
\label{sec:records}
We say that $n\in\bbN$ is a record if $n=1$ or $n\ge 2$ and $T_n > \max\{T_1, \ldots, T_{n-1}\}$.
Let us define
\beq
I_n = 
\begin{cases}
1 & \text{if $n$ is a record}\\
0 & \text{otherwise},
\end{cases}
\eeq
the indicator of a record, and
\beq
\mathfrak{R}_n = \sum_{k=1}^n I_k = |R(1,n)|
\eeq
the number of records between $1$ and $n$. It was noticed (see Rényi~\cite{Re62}) that when $T_1$ is a continuous random variable, the $I_n$'s are independent Bernoulli random variables with mean $\bE(I_n) = 1/n$. However, we deal here with the discrete case, where this simple structure breaks down because of the possibility of ties. Actually, this case seems to have attracted relatively little attention (see however~\cite{GLS} and references therein). In this section, we provide some results in the discrete case (moments and deviations from the mean for the number of records) that will be useful later, though we do not aim for optimality. We start with:
\begin{proposition}
\label{pr:records2}
For all $p\in\bbN$ and $1\le n_1 < n_2 < \ldots n_p$,
\beq
\bE(I_{n_1}\ldots I_{n_p}) \le (1/n_1)\ldots (1/n_p).
\eeq
\end{proposition}
As a consequence we obtain:
\begin{proposition}
\label{lem:records_dev}
For $b>1$ there exists a positive constant $c=c(b)$ (which we may choose smaller but arbitrarily close to $\sup_{\gl>0}\{b\gl + 1 - e^{\gl}\}=1+b(\ln b -1) >0$) such that for $n$ large enough,
\beq
\label{eq:records_dev2}
\bP(\mathfrak{R}_n \ge b\log n) \le n^{-c}.
\eeq
\end{proposition}

\begin{proof}[Proof of Proposition \ref{pr:records2}]
We prove it by iteration on $p\ge 1$ and thus start with the case $p=1$. Let $n\ge 2$ (the statement is trivial for $n=1)$. Let $J_n$ be the indicator of the event that there is a strict maximum among the $n$ first variables, meaning that
\beq
\label{eq:Jn}
J_n = \ind_{\{\exists 1\le i \le n \colon \forall  1\le j \le n,\ j\neq i,\ T_i > T_j\}}.
\eeq
By exchangeability of the $n$ first variables we get that 
\beq
\label{eq:Jn2}
\bE(I_n) = \frac1n \bE(J_n) \le \frac1n.
\eeq
Suppose now that the statement is true for $p\ge1$ and let $2\le n_1 < \ldots < n_{p+1}$ (if $n_1= 1$, we can safely remove it). Again, by exchangeability of $(T_1, \ldots, T_{n_1})$ we have
\beq
\bE(I_{n_1}\ldots I_{n_{p+1}}) = \frac{1}{n_1} \bE({J_{n_1}}I_{n_2}\ldots I_{n_{p+1}}) \le \frac{1}{n_1} \bE(I_{n_2}\ldots I_{n_{p+1}}),
\eeq
and the result follows by using the induction hypothesis.
\end{proof}
\begin{proof}[Proof of Proposition \ref{lem:records_dev}]
Let $\gl>0$ to be specified later. By Chernoff's bound,
\beq
\bP(\mathfrak{R}_n \ge b\log n) \le e^{-b\gl \log n}\bE[e^{\gl R_n}].
\eeq
Since $\mathfrak{R}_n = \sum_{k=1}^n I_k$ and the $I_k$'s are $\{0,1\}$valued random variables, we get
\beq
e^{\gl \mathfrak{R}_n} = \prod_{k=1}^n (1+ [e^\gl-1]I_k) = 1 + \sumtwo{J\subseteq \{1,\ldots, n\}}{J\neq\eset} [e^{\gl}-1]^{|J|} \prod_{j\in J} I_j.
\eeq
By taking the expectation and using Proposition~\ref{pr:records2}, we get
\beq
\bE[e^{\gl \mathfrak{R}_n}] \le \prod_{k=1}^n \Big(1+ [e^\gl-1]\frac1k\Big) \le \exp([e^\gl-1][1+o(1)]\log n ),
\eeq
where the $o(1)$ holds as $n\to\infty$. Finally, we obtain 
\beq
\bP(\mathfrak{R}_n \ge b\log n) \le \exp(\{[e^\gl-1][1+o(1)]-b\gl\}\log n ),
\eeq
which concludes the proof.
\end{proof}

\subsubsection{Proof of Proposition~\ref{lem:ge}} \label{sec:proba.good.env}

We will use notations from the previous subsection during the proof, and sometimes write $\mathfrak{R}(n)$ instead of $\mathfrak{R}_n$ for the sake of readability. We consider the events $A_n^{(i)}$ for $1\le i \le 11$ and conclude by combining the results obtained in each cases.\\
Along the proof we will use that, by~\eqref{eq:tail.ass}, there exists $c_1,c_2>0$ such that for all $m\in \bbN$,
\beq
\label{eq:UB_tail_T}
1 - c_1m^{-\gamma} \leq \bbP(T_1 \le m) \le 1 - c_2m^{-\gamma} \le \exp(-c_2m^{-\gamma}).
\eeq
{\noindent \it {Case $i=1$}}. Assume first that $\gamma \le1$, so that $\gk = \frac{\gamma}{4}$. 
Since 
\beq
N^{(1+\frac{\gamma}{4})\frac{2}{\gamma}} = o(n^{1-\frac{\gamma}{4(\gamma+2)}}),
\eeq
one has $\lim_{n\to\8}\bbP(A_n^{(1)}) = 0$ from Proposition~\ref{thm:limit.stable}. Assume now that $\gamma>1$. Then, 
\beq
N^{2(1+\gk)} = 
\begin{cases}
o(n^{1-\frac{2-\gamma}{4(2+\gamma)}}) & \text{if } \gga\in(1,2)\\
o(n^{1+\frac{2\gga-1}{2(\gga+2)}}) & \text{if } \gga \ge2,
\end{cases}
\eeq
and again $\lim_{n\to\8}\bbP(A_n^{(1)}) = 0$ from Proposition~\ref{thm:limit.stable}.\\

{\noindent \it {Case $i=2$}}. Note that, by \eqref{eq:UB_tail_T},
\beq
\ba
 \bbP\left(A_n^{(2)}(\gep_0)\right) =  \bbP\left(T_1 \leq  \gep_0^{\frac{1}{2\gamma}}N^{1/ \gamma}\right)^{\gep_0 N}\geq \left(1-\frac{c_1}{\gep_0^{\frac{1}{2}}N}\right)^{\gep_0 N}=e^{-c_1\gep_0^{1/2}}+o(1)\qquad n\to +\8.
\ea
\eeq
We obtain $\lim_{\gep_0\to 0} \liminf_{n \to +\8} \bbP(A_n^{(2)}(\gep_0))=1$.\\

{\noindent \it {Case $i=3$}}. Here we note that $N^{1/\gamma} = o(n)$ when $\gamma\le1$ and $N = o(n)$ when $\gamma>1$, and we conclude using Proposition \ref{thm:limit.stable} that for all $\gep_0>0$
\beq
\lim_{n\to\8}\bbP(A_n^{(3)}(\gep_0)) = 0.
\eeq

{\noindent \it {Case $i=4$}.} By independence of the $T_\ell$'s, one has
\beq
\bbP(A_n^{(4)}(\gd,\gep)^c) \le \bbP(T_1 < T_0(\gep) \vee \gd N^{\frac{1}{\gamma}})^N.
\eeq
Therefore,
\beq
\bbP(A_n^{(4)}(\gd,\gep)^c) \le \exp( - c(NT_0(\gep)^{-\gamma} \wedge \gd^{-\gamma})),
\eeq
from which we deduce that for all $\gep>0$
\beq
\label{eq:cas4}
\lim_{\gd\to 0} \liminf_{n\to\8} \bbP(A_n^{(4)}(\gd,\gep)) = 1.
\eeq

{\noindent \it {Case $i=5$}.} Recall the definitions of $R_{\gep_0}(n)$ and $T^*_k$ in~\eqref{eq:def.t.tau.star} and~\eqref{eq:def.R.cR}. We get
\beq
\bbP(A_n^{(5)}(\gep_0,\gep)^c) \le \bbP(\forall \ell \le \gep_0N,\ T_\ell \leq T_0(\gep) \vee \gep_0^{\frac{3}{2\gamma}}N^{\frac{1}{\gamma}}) \le \bbP(T_1 \leq T_0(\gep) \vee \gep_0^{\frac{3}{2\gamma}}N^{\frac{1}{\gamma}})^{\gep_0 N}.
\eeq
By using again~\eqref{eq:UB_tail_T}, we obtain
\beq
\bbP(A_n^{(5)}(\gep_0,\gep)^c) \le \exp(-c_2 (\gep_0T_0(\gep)^{-\gamma}N \wedge \gep_0^{-1/2})),
\eeq
from which we get for all $\gep>0$
\beq
\label{eq:cas5}
\lim_{\gep_0\to 0} \liminf_{n\to\8} \bbP(A_n^{(5)}(\gep_0,\gep)) = 1.
\eeq

{\noindent \it {Case $6\le i\le 8$}}. Let us first prove that for $\gep_0 >0$,
\beq
\label{eq:cas8}
\lim_{n\to\infty} \bbP(A_n^{(8)}(\gep_0)) =1.
\eeq
Indeed, since $|R(1,N/\gep_0)| = \mathfrak{R}(N/\gep_0)$, it is enough to prove that
\beq
\bbP(\mathfrak{R}_n \le (\log n)^2) \to 1,\qquad n\to\8,
\eeq
which can be easily deduced from Proposition~\ref{lem:records_dev}. We now deal with the event corresponding to $i=7$.
Fix $\gep_0,\gep>0$, and note that
\beq
\label{eq:cas7aux1}
\ba
\bbP(A_n^{(7)}(\gep_0,\gep)^c) \le \bbP( & A_n^{(8)}(\gep_0)^c) + \bbP(A_n^{(5)}(\gep_0,\gep)^c)\\
& +\bbP \Big(\exists k\leq [\log(N/\gep_0)]^2 \colon \gep_0^{\frac{3}{2\gamma}}N^{\frac{1}{\gamma}} \le T^*_k \le h(f_k,L_k,\ga(\gep)),\ i(k)\le N/\gep_0 \Big).
\ea
\eeq
By a union bound we are left to prove that 
\beq
\label{eq:but}
p_k(n,\gep_0) = o([\log n]^{-2}),
\eeq
where
\beq
\label{eq:628}
p_k(n,\gep_0,\gep) := \bbP(\gep_0^{\frac{3}{2\gamma}}N^{\frac{1}{\gamma}} \le T^*_k \le h(f_k,L_k,\ga(\gep)),\ i(k)\le N/\gep_0).
\eeq
From~\eqref{eq:def.gALga} in the proof of Lemma~\ref{lem:control_pinbad}, we have for $n$ large enough,
\beq
\label{eq:pk1pk2}
p_k(n,\gep_0,\gep) \le \bbP(f_k^{L_k} \ge C(\gep_0,\gep) N^{\frac{1}{2\gamma}} ,\ T_k^* \ge \gep_0^{\frac{3}{2\gamma}}N^{\frac{1}{\gamma}},\ i(k)\le N/\gep_0).
\eeq
Let us first show that
\beq
\label{eq:LBfk}
\bbP\Big(f_k > A,\ T_k^* \ge \gep_0^{\frac{3}{2\gamma}}N^{\frac{1}{\gamma}}\Big) \le 
C \Big[A \wedge \gep_0^{\frac{3}{\gamma}}N^{\frac{2}{\gamma}}\Big]^{-1}.
\eeq
To this end, our first ingredient is the following upper bound, which holds for $u\in(0,1)$ and $k\in\bbN$:
\beq
\label{eq:ratio_rec_UB}
\bbP\Big(\frac{T^*_{k}}{T^*_{k+1}}\ge1-u\Big) \le C u.
\eeq
We show this inequality at the end of the proof, in order not to disrupt the main line of ideas.
We will also use the following elementary bound:
\beq
\label{eq:ineqf}
f(u)=\frac{u}{\sin(u)}\leq \frac{C}{\pi-u}, \qquad u\in (0,\pi).
\eeq
Coming back to \eqref{eq:LBfk}, we have by using~\eqref{eq:ineqf},
\beq
\ba
\bbP\Big(f_k > A,\ T_k^* \ge \gep_0^{\frac{3}{2\gamma}}N^{\frac{1}{\gamma}}\Big) &\le 
\bbP\Big( \frac{T_{k-1}^*}{T_{k}^*} \Big[1+\frac{C}{(T_{k}^*)^2}\Big] \ge 1 - \frac{C}{A},\ T_k^* \ge \gep_0^{\frac{3}{2\gamma}}N^{\frac{1}{\gamma}}\Big)\\
& \le \bbP\Big( \frac{T_{k-1}^*}{T_{k}^*}\ge 1 - \frac{C}{A} - \frac{C}{\gep_0^{\frac{3}{\gamma}}N^{\frac{2}{\gamma}}}\Big),
\ea
\eeq
and we obtain~\eqref{eq:LBfk} thanks to~\eqref{eq:ratio_rec_UB}. In view of \eqref{eq:but}, we choose $A=A_n=[\log n]^3$. For $k\geq 1$, from \eqref{eq:pk1pk2},
\beq
\label{eq:634}
p_k(n,\gep_0,\gep)\leq \bbP \Big(f_k \ge A_n ,\ T_k^* \ge \gep_0^{\frac{3}{2\gamma}}N^{\frac{1}{\gamma}}\Big)
+\bbP\Big(L_k \ge C \frac{\log n}{\log \log n},\ T_k^* \ge \gep_0^{\frac{3}{2\gamma}}N^{\frac{1}{\gamma}},\ i(k)\le N/\gep_0\Big).
\eeq
Using \eqref{eq:LBfk}, the first term in the sum above is $o[(\log n)^{-2}]$. We now deal with the second one. 
On the corresponding event, we have
\beq
\label{eq:event.bin}
\card\Big\{1\le j \le N/\gep_0 \colon T_j \ge \ga \gep_0^{\frac{3}{2\gamma}}N^{\frac{1}{\gamma}} \Big\} \ge C \frac{\log n} {\log \log n}.
\eeq
Furthermore, the random variable in the l.h.s. of the inequality follows a binomial law with parameters $N/\gep_0$ and $\bbP(T_1 \ge  \ga \gep_0^{\frac{3}{2\gamma}}N^{\frac{1}{\gamma}}) \le \alpha^{-\gamma}\gep_0^{-3/2}N^{-1}$ (up to a harmless constant). By using a standard binomial-to-Poisson approximation (cf.\ end of the proof)
\beq
\label{eq:approx.bin}
\bbP(\bin(\ell,q)\ge m) \le \Big( \frac{q\ell}{m} \Big)^{m} e^{m-q\ell},\qquad \ell\in\bbN,\ q\in(0,1),\ m\in\bbN \colon q\ell < m,
\eeq
with $q=\gep_0^{-3/2}N^{-1}$, $\ell = N/\gep_0$ and $m = C \log n/\log \log n$, 
we get that
\beq
\label{eq:LB.Lk}
\bbP(\text{\eqref{eq:event.bin}}) = o(n^{-C/2}),
\eeq
which is enough to conclude. From what precedes we finally obtain that for all $\gep_0,\gep>0$,
\beq
\lim_{n\to\infty} \bbP(A_n^{(7)}(\gep_0,\gep)^c) = 0.
\eeq 
The event corresponding to $i=6$ can be readily treated with the same idea, since the $N^{\frac{1}{2\gamma}}$ in \eqref{eq:pk1pk2} is less than $\exp(n^{\frac{\gamma}{2(\gamma+2)}})$. Finally,
\beq
\label{eq:cas67}
\lim_{n\to\infty} \bbP(A_n^{(6)}(\gep_0,\gep)^c \cup A_n^{(7)}(\gep_0,\gep)^c) = 0.
\eeq

{\noindent \it {Case $9\le i\le 10$}}. From the almost-sure convergence in Proposition~\ref{pr:lambda.beta}, we readily get that for any choice of $\gd,\gep_0,\gep,\eta$,
\beq
\label{eq:cas910}
\lim_{n\to \infty} \bbP(A_n^{(9)}(\gd) \cap A_n^{(10)}(\gep_0,\gep,\eta)) = 1.
\eeq

{\noindent \it {Case $i=11$}}. Note that by \eqref{eq:UB_tail_T}, 
\beq
\ba
 \bbP\left(A_n^{(11)}(\gep_0)\right) = \bbP\left(T_1 \leq  \frac{N^{1/ \gamma}}{\gep_0}\right)^N \geq e^{-c_1 \gep_0^\gamma}+o(1) \qquad n \to +\8 .
\ea
\eeq
We obtain $\lim_{\gep_0\to 0} \liminf_{n \to +\8} \bbP(A_n^{(11)}(\gep_0))=1$.

{\noindent \it Proof of \eqref{eq:ratio_rec_UB}.}
By writing $1+v = (1-u)^{-1}$ for convenience, one has
\beq
\bbP\Big(T^*_{k+1} \le (1+v){T^*_k} \Big|\ T^*_k = \ell \Big) 
=
\begin{cases}
0 & \text{if } v < 1/\ell\\
\bbP(T_1 \le \lfloor(1+v)\ell\rfloor\ |\ T_1 > \ell) & \text{otherwise}.
\end{cases}
\eeq
From our assumption on the tail of $T_1$ and with the help of a standard comparison between series and integrals, the probability in the second case is bounded from above by
\beq
C\ell^{\gamma} \sum_{n=\ell+1}^{\lfloor(1+v)\ell\rfloor} n^{-(1+\gamma)} \le C\ell^{\gamma} \int_{\ell}^{\lfloor(1+v)\ell\rfloor-1} t^{-(1+\gamma)}\dd t.
\eeq
In turn, the integral above is controlled by
\beq
\int_{\ell}^{(1+v)\ell} t^{-(1+\gamma)}\dd t = C[1-(1+v)^{-\gamma}]\ell^{-\gamma} \le C u \ell^{-\gamma},
\eeq
which completes the proof of \eqref{eq:ratio_rec_UB}.\\

{\noindent \it Proof of \eqref{eq:approx.bin}.} By Chernoff's bound, one has for $x>0$,
\beq
\bbP(\bin(\ell,q)\ge m) \le e^{-x m} \bbE(e^{x \ber(q)})^\ell \le \exp(-x m+ q(e^x -1)\ell).
\eeq  
Since $m> q\ell$, we may choose $x = \log(m/(q\ell))$ to minimize the r.h.s.\ in the line above and get the result.

\subsection{Conclusion : proof of Theorem \ref{thm0}} \label{subsec:conclusion}
To prove Theorem \ref{thm0}, we establish that for all $u\in \bbR$ which is a continuity point of the distribution function of $F$,
\beq
\label{main_eq}
\lim_{n\to\infty} \bbP(F_n \le u) = \bbP(F \le u).
\eeq
By \eqref{eq:Fcont} all real numbers are continuity points of $F$. Moreover since $F_n$ is positive, we only have to prove \eqref{main_eq} for $u>0$.

\par  We start with the \textbf{upper bound} in \eqref{main_eq}
\beq
\label{main_eqUB}
\limsup_{n\to\infty} \bbP(F_n \le u) \leq \bbP(F \le u).
\eeq
Fix $\theta>0$ that we will let go to $0$ only at the very end of the proof. Fix also $\gep,\eta>0$. From Proposition~\ref{lem:ge}, there exists $\delta>0$ and $\gep_1(\delta)$ so that for $\gep_0<\gep_1(\delta)$
\beq
\label{eq:choice}
\liminf_{n}\bbP(\Omega_n(\delta,\gep_0,\gep,\eta))>1-\theta.
\eeq
Fix $\gep_0<\gep_1(\delta)$ small enough so that the conclusion of Proposition \ref{prop:ub} is satisfied. Thus we obtain that for $n$ large enough
\beq
\label{eq:grossier}
\bbP (F_n \leq u) \leq \bbP(\min_{\ell\in \mathcal{R}_{\gep_0}(n) } G_n^{\beta-\epsilon}(\ell)\leq u + \eta, \Omega_n) +\bbP(\Omega_n^c).
\eeq
From Proposition~\ref{lem:ge} and the choices of $\delta$ and $\gep_0$, $\bbP(\Omega_n^c)< \theta $ for $n$ large enough. Thus we just have to focus on the first term in the last equation. 
We introduce for $n,\ell\geq 1$ and $\beta>0$, the random variable 
\beq
\tilde{G}^{\beta}_n(\ell)=\frac{\lambda(\beta) (\ell-1)}{N}+\frac{\pi^2 n}{2 T_\ell^2 N}= \psi^{\lambda(\beta)}\left(\frac{\ell}{N}, \frac{T_{\ell}}{N^{1/\gamma}}\right).
\eeq
We replace $G$ by $\tilde{G}$ in \eqref{eq:grossier} and control the probability that both processes are not close
 \beq
 \label{eq:mainub}
 \ba
\bbP(\min_{\ell\in \mathcal{R}_{\gep_0}(n) } &G_n^{\beta-\epsilon}(\ell) \leq u + \eta, \Omega_n) \\
&\leq \bbP(\min_{\ell \in \cR} \tilde{G}_n^{\beta-\epsilon}(\ell)< u +2 \eta, \Omega_n)+\bbP(\max_{\ell\in \cR_{\gep_0}(n)} |\tilde{G}_n^{\beta-\epsilon}(\ell)-G_n^{\beta-\epsilon}(\ell)| \geq \eta, \Omega_n).
\ea
\eeq
The first term in the sum gives the main contribution. Let us first prove that the second one is zero for $n$ large enough.
For $\ell\geq 1$ we define
\beq
\Delta_1(n,\ell):=\frac{\ell-1}{N}\left|\lambda (\beta-\epsilon) - \gl(\ell-1,\gb-\gep)\right|
\eeq
and 
\beq
\Delta_2(n,\ell):=\frac{n}{N}|g(T_\ell) -\frac{\pi^2}{2 T_\ell^2}|
\eeq
so that 
\beq
\max_{\ell\in \cR_{\gep_0}(n)}  |\tilde{G}_n^{\beta-\epsilon}(\ell)-G_n^{\beta-\epsilon}(\ell)| \leq \max_{\ell\in \cR_{\gep_0}(n)}  \Delta_1(n,\ell) + \max_{\ell\in \cR_{\gep_0}(n)}  \Delta_2(n,\ell).
\eeq
We first deal with $\Delta_2$. According to \eqref{eq:g} there exists some $C>0$ such that for all $\ell$, $|g(T_\ell) -\frac{\pi^2}{2 T_\ell^2}| \leq \frac{C}{T_\ell^4} $. 
 We can thus deduce that
\beq
\bbP( \max_{\ell\in \cR_{\gep_0}(n)}  \Delta_2(n,\ell) \geq \eta,  \Omega_n) \leq \bbP\left(\frac{C\ n}{N (\gep_0^{\frac{3}{2 \gamma }}N^{1/\gamma})^4} \geq \eta , A^{(5)}_n(\gep_0,\gep) \right),
\eeq
and this last term is $0$ for $n$ large enough.
We turn to the control of $\Delta_1$. 
For $n\ge 1$,
\beq
\ba
\label{eq:useAn10}
\bbP(\max_{\ell\in \cR_{\gep_0}(n)}  \Delta_1(n,\ell) \geq \eta,  \Omega_n) &\leq \bbP(\max_{\ell\in \cR_{\gep_0}(n)}  \Delta_1(n,\ell) \geq \eta, A_n^{(10)}(\gep_0,\gep,\eta)),
\ea
\eeq
and again the last term is $0$ for $n$ large enough.
Let us come back to the first term in \eqref{eq:mainub}, $\bbP(\min_{\ell \in \cR} \tilde{G}_n^{\beta-\epsilon}(\ell)< u +2 \eta, \Omega_n)$.
As $\ell$ ranges $\cR$ we may write 
\beq
\min_{\ell  \in \cR} \tilde{G}_n^{\beta-\epsilon}(\ell)=\Psi^{\lambda(\beta-\epsilon)} \left(\Pi_N\right).
\eeq
We thus obtain
\beq
\label{eq:ubGtilde}
\ba
\bbP(\min_{\ell \in \cR }\  \tilde{G}_n^{\beta-\epsilon}(\ell)< u +2 \eta, \Omega_n ) &\leq \bbP(\Psi^{\lambda(\beta-\epsilon)} \left(\Pi_N\right)< u +2 \eta, A^{(11)}_n(\gep_0) )\\
& \leq \bbP(\Psi^{\lambda(\beta-\epsilon)}_K \left(\Pi_N\right) < u +2 \eta),
\ea
\eeq
where $K:=A^{\lambda(\beta-\gep)}_{u+2\eta} \cap \{y\leq 1/\gep_0\}$ with the set $A$ defined in \eqref{eq:defA}. As $K$ is compact, Proposition \ref{prop:convergenceP} and Lemma \ref{prop:continuity} assure that
\beq
\bbP(\Psi_K^{\lambda(\beta-\epsilon)} \left(\Pi_N\right) < u +2 \eta) \to \bbP(\Psi^{\lambda(\beta-\epsilon)}_K(\Pi)< u+2\eta)
\eeq
when $n$ goes to infinity (we recall that $\Psi^{\lambda(\beta-\epsilon)}_K(\Pi)$ is continuous).
Using that $\Psi^{\lambda(\beta-\epsilon)} \leq \Psi^{\lambda(\beta-\epsilon)}_K$ we obtain 
\beq
\ba
\bbP(\Psi^{\lambda(\beta-\epsilon)}_K(\Pi)< u+2\eta) &\leq \bbP(\Psi^{\lambda(\beta-\epsilon)}(\Pi)< u+2\eta)\\
										&=\bbP(F^{\beta-\epsilon}< u+2\eta).
\ea
\eeq
Finally, we have proven
\beq
\limsup_{n\to +\8}\bbP (F_n \leq u) \leq \bbP(F^{\beta-\epsilon}< u+2\eta)+{  \theta}.
\eeq
As $u\mapsto \bbP(F^{\beta-\epsilon}\leq u)$ is right-continuous,
\beq
\label{eq:presk}
\lim_{\eta \to 0} \bbP(F^{\beta-\epsilon}< u+2\eta) =  \bbP(F^{\beta-\epsilon}\leq u). 
\eeq
From Lemma \ref{lem:Fgep}, $F^{\beta-\epsilon} {\nearrow}  F^{\beta}$ almost surely when $\gep$ goes to $0$ so that
\beq
 \bbP(F^{\beta-\epsilon}\leq u){\to}  \bbP(F^{\beta}\leq u) \qquad  \epsilon \to 0.
\eeq
Finally,
\beq
\limsup_{n\to +\8}\bbP (F_n \leq u) \leq \bbP(F^{\beta}\leq u)+{  \theta},
\eeq
and, as $\theta$ can be chosen arbitrarily small, we obtain the upper bound
\beq
\limsup_{n\to +\8}\bbP (F_n \leq u) \leq \bbP(F^{\beta}\leq u).
\eeq
 \par We turn now to the \textbf{lower bound} in \eqref{main_eq}:
\beq
\liminf_{n\to\infty} \bbP(F_n \le u) \geq \bbP(F \le u),
\eeq
or, equivalently,
\beq
\limsup_{n\to\infty} \bbP(F_n > u) \leq \bbP(F > u).
\eeq 
The proof works essentially in the same way as for the upper bound. Again fix $\theta,\gep,\eta>0$ and, using Proposition~ \ref{lem:ge}, $\delta>0$ and $\gep_1(\delta)$ so that for $\gep_0<\gep_1(\delta)$
\beq
\label{eq:choice}
\liminf_{n}\bbP(\Omega_n(\delta,\gep_0,\gep,\eta))>1-\theta.
\eeq
We choose $\gep_0<\gep_1(\delta)$ small enough so that 

\begin{enumerate}
\item the conclusion of Proposition \ref{prop:lb} is satisfied;
\item the following inequality holds
\beq
\label{eq:condlb}
\frac{\lambda(\beta)}{\gep_0}> 2\lambda(\beta)+ \frac{\pi^2}{2\delta^2}.
\eeq 
\end{enumerate} 
Using Proposition \ref{prop:lb}, for $n$ large enough,
\beq 
\label{eq:presque}
\ba
\bbP(F_n > u) &\leq \bbP(\min_{1< \ell \le N^{1+\gk}} G^{\beta}_n(\ell)> u-\eta, \Omega_n) + \bbP(\Omega_n^c)\\
		&\leq \bbP(\min_{\ell \in \cR_{\gep_0}(n)} G^{\beta}_n(\ell)> u-\eta, \Omega_n) + \bbP(\Omega_n^c) \\
			&\leq \bbP(\min_{\ell \in \cR_{\gep_0}(n)} \tilde{G}^{\beta}_n(\ell)> u-2\eta, \Omega_n)+\bbP(\max_{\ell \in \cR_{\gep_0}(n) } |\tilde{G}_n^{\beta}(\ell)-G^{\beta}_n(\ell)| \geq \eta,  \Omega_n)+\bbP(\Omega_n^c).
\ea
\eeq
The second term in this last equation is treated exactly in the same way as the second term in \eqref{eq:mainub} and is thus zero for $n$ large enough. The third one is smaller than $\theta$ by \eqref{eq:choice} for $n$ large enough. We thus focus on the first one. The choice of $\gep_0$ in \eqref{eq:condlb} implies that 
\beq
\Omega_n\subset \{\argmin \tilde{G}_n < N/\gep_0\}.
\eeq
Indeed, as $\gO_n\subset A^{(4)}_n(\delta,\gep)$, it holds that, on $\gO_n$,
\beq
\min_{N\leq \ell \leq 2N} \tilde{G}^{\beta}_n < 2\lambda(\beta) + \frac{\pi^2}{2\delta^2}<\frac{\lambda(\beta)}{\gep_0}<\min_{\ell>N/\gep_0}\tilde{G}^{\beta}_n(\ell).
\eeq
Therefore, for any compact set $K$ in $E$,
\beq
\ba
\bbP(\min_{\ell \in \cR_{\gep_0}(n)} \tilde{G}^{\beta}_n(\ell)> u-2\eta, \Omega_n) &=  \bbP(\min_{\ell \in \cR} \tilde{G}^{\beta}_n(\ell)> u-2\eta,  \Omega_n)\\
&=\bbP(\Psi^{\lambda(\beta)}(\Pi_N)> u-2\eta,\Omega_n)\\
&\leq \bbP(\Psi_K^{\lambda(\beta)}(\Pi_N)> u-2\eta).
\ea
\eeq
By Lemma \ref{prop:continuity}, $\bbP(\Psi_K^{\lambda(\beta)}(\Pi_N)> u-2\eta)$ converges to $\bbP(\Psi_K^{\lambda(\beta)}(\Pi)> u-2\eta)$ when $N$ goes to infinity. 
Finally,
\beq
\limsup_{n\to\infty} \bbP(F_n > u) \leq \bbP(\Psi_K^{\lambda(\beta)}(\Pi)> u-2\eta) + \theta.
\eeq
By letting $K$ increase to $E$, we obtain 
\beq
\limsup_{n\to\infty} \bbP(F_n > u) \leq \bbP(\Psi^{\lambda(\beta)}(\Pi)> u-2\eta) + \theta,
\eeq
and we conclude as for the upper bound by letting $\eta$ and $\theta$ go to $0$. 

\appendix

\section{Proof of Proposition \ref{pr:periodic}}
\label{sec:periodic}

The proof is divided into several steps.
In the following, we partition $\tau$ into $p$ disjoint subsets $\tau^{(i)} = \tau_i + \tau_p\bbZ$, for $0\le i < p$. \\

{\bf \noindent Step 1. Decomposition of the probability.}
We first consider the event $\{\gs > n,\ S_n\in \tau\}$ instead of $\{\gs > n\}$ and will come back to the original event at the final step. 
By decomposing according to the visits to $\tau$ and by using the Markov property, we obtain
\beq
\label{eq:roughUBstep2}
\bP(\gs > n,\ S_n\in \tau) =  \sum_{m=1}^n \sum_{0<u_1<\ldots<u_m=n} \sum_{x_1,\ldots, x_m} \prod_{i=1}^m \Big( q_{x_{i-1}, x_i}(u_i - u_{i-1}) e^{-\gb} \Big),
\eeq
where $u_0 = 0$, $x_0=0$,  $x_1,\ldots, x_m \in \bbZ/p\bbZ$, and the $q_{ij}(n)$'s are a slight modification of the ones defined in~\eqref{eq:defqij}, namely
\beq
\label{eq:def_qij}
q_{ij}(n) = \bP_{\tau_i}(S_k \notin \tau,\ 1\le k < n,\ S_n \in \tau^{(j)}),\qquad i,j\in \bbZ/p\bbZ, \quad n\ge 1.
\eeq
It will be helpful later in the proof to know the asymptotic behaviour of $q_{ij}(n)$, as $n\to\infty$:
\beq
\lim \frac{1}{n} \log q_{ij}(n) = - g(t_{ij}), \qquad n\to \infty,
\eeq
where $g$ has been defined in~\eqref{eq:g} and $t_{ij}$ in~\eqref{eq:deftij}.\\

{\bf \noindent Step 2. Definition of $\phi(\gb ; t_1, \ldots, t_p)$.}
Recall the definition of $Q_{ij}(\phi)$ in~\eqref{eq:defQij} (with $q_{ij}$ defined in \eqref{eq:def_qij}), which is now restricted to $i,j\in \bbZ/p\bbZ$.
For all $i,j$, $Q_{ij}(\phi)$ is finite and increasing on $[0, g(t_{ij}))$, infinite on $[g(t_{ij}),\infty)$ and its limit at $g(t_{ij})$ is infinite, by Proposition \ref{pr:small_ball}. Let $\gL(\phi)$ be the Perron-Frobenius eigenvalue of $Q(\phi)$, defined as infinity when one of the entry is infinite, that is for $\phi \ge \min g(t_{ij}) = g(t_{\max})$. 
We recall that
\beq
\gL(\phi) = \sup_{v\neq 0} \min_i \frac{(Q(\phi) v)_i}{v_i}.
\eeq
From what precedes,
$\gL$ is increasing on $[0, g(t_{\max}))$ and tends to $\infty$ as $\phi \nearrow g(t_{\max})$. Therefore, the equation
\beq
\gL(\phi) = \exp(\gb)
\eeq
has a unique positive solution on this interval, that we denote by $\phi(\gb ; t_1, \ldots, t_p)$. As $\phi(\gb ; t_1, \ldots, t_p)\in [0, g(t_{\max}))$, this proves the second inequality in \eqref{eq:comp.phi}.
In the sequel of the proof, for the sake of conciseness, we use the notation  $\phi(\gb)=\phi(\gb ; t_1, \ldots, t_p)$. 
Coming back to \eqref{eq:roughUBstep2}, we get
\beq
\label{eq:roughUBstep2b}
\text{ r.h.s\eqref{eq:roughUBstep2}} = e^{-\phi(\gb)n}
\sum_{m=1}^n \sum_{0<u_1<\ldots<u_m=n} \sum_{x_1, \ldots, x_m} \prod_{i=1}^m \Big( q_{x_{i-1}, x_i}(u_i - u_{i-1}) e^{-\gb+\phi(\gb)(u_i - u_{i-1})} \Big).
\eeq
{\bf \noindent Step 3. Spectral decomposition and a first upper bound.}
The key idea in this step is a spectral decomposition, which is a technique used also in the context of the parabolic Anderson model, see \cite[Section $2.2.1$]{Ko16}.
Let us define a matrix $Q^\gb$ by
\beq 
Q^{\gb}_{ij} = Q_{ij}(\phi(\gb))e^{-\gb},\qquad i,j\in \bbZ/p\bbZ,
\eeq
which is symmetric (by symmetry of the simple random walk) with positive entries. From what precedes, its top eigenvalue is one. Therefore, we may denote by $1= \gl_0 \ge \gl_1 \ge \ldots  \ge \gl_{p-1}$ its eigenvalues by decreasing order (with possible repetitions), with $|\gl_k|<1$ for $k>1$ (Theorem 1.1 in Seneta~\cite{S06}) and $(\nu_i)_{0\le i < p}$ an associated basis of orthonormal left eigenvectors. Note that for all $0\le x < p$, we have $\gd_x = \sum_{i=0}^{p-1} \nu_i(x) \nu_i$, that is the element of $\bbR^p$ which is $1$ at coordinate $x$ and $0$ elsewhere. Let us now define, for $0\le a < p$,
\beq
\cZ_n(a)=\sum_{m=1}^n  \sum_{x_0 = a, x_1, \ldots, x_m} \prod_{i=1}^m Q^{\gb}_{x_{i-1}, x_i}.
\eeq
By removing the condition $\{u_m = n\}$ in \eqref{eq:roughUBstep2b}, we get the upper bound
\beq
\text{r.h.s.\eqref{eq:roughUBstep2b}} \le e^{-\phi(\gb)n} \cZ_n(0).
\eeq
Moreover,
\beq
\cZ_n(0) = \langle \gd_0,\ \cZ_n(\cdot) \rangle =  \sum_{i=0}^{p-1} \sum_{j=0}^{p-1} \nu_i(0) \nu_i(j) \cZ_n(j),
\eeq
(with the usual scalar product) which yields
\beq
\cZ_n(0) = \sum_{i=0}^{p-1} \sum_{j=0}^{p-1} \sum_{m=1}^n \nu_i(0) \nu_i(j) \sum_{x_0 = j, x_1 \ldots, x_m} \prod_{k=1}^m Q^{\gb}_{x_{k-1}, x_k}.
\eeq
By definition of the $\nu_i$'s we get for $0\le i < p$,
\beq
\sum_{j=0}^{p-1} \nu_i(j) \sum_{x_0 = j, x_1, \ldots, x_m} \prod_{k=1}^m Q^{\gb}_{x_{k-1}, x_k} 
= \nu_i (Q^\gb)^m \ind
= \gl_i^m \sum_{j=0}^{p-1} \nu_i(j),
\eeq
where $\ind$ is the vector with all one.
Therefore,
\beq
\cZ_n(0) = \sum_{i=0}^{p-1} \sum_{j=0}^{p-1} \sum_{m=1}^n \nu_i(0) \nu_i(j) \gl_i^m
 \le n \sum_{i=0}^{p-1} \sum_{j=0}^{p-1} |\nu_i(0)| |\nu_i(j)|
 \le np,
\eeq
where in the first inequality we use that $|\gl_i|\le 1$ and the triangular inequality, while in the second inequality we use the Cauchy-Schwarz inequality and the fact that
\beq
\sum_{j=0}^{p-1} \nu_i(j)^2 = \|\nu_i\|_2^2 = 1,\quad \sum_{i=0}^{p-1} \nu_i(0)^2 = \|\gd_0\|_2^2 = 1.
\eeq
Finally, we have obtained
\beq
\label{eq:A15}
\bP(\gs > n,\ S_n\in \tau) \le np e^{-\phi(\gb) n}.
\eeq

{\bf \noindent Step 4. Lower bound in \eqref{eq:315}.} The components of a Perron-Frobenius eigenvector being all positive (or all negative), we may consider the matrix $\{Q^{\gb}_{ij}\frac{\nu_0(j)}{\nu_0(i)} \}_{ij}$, which turns out to be stochastic. This actually defines a Markov renewal process $\rho$  on  $\bbZ/p\bbZ$ (see e.g. Section VII.4 in Asmussen~\cite{A03}) with law $\cP_\gb$ determined by the kernel
\beq
q^\gb_{ij}(n) = \exp(\phi(\gb) n - \gb) q_{ij}(n) \frac{\nu_0(j)}{\nu_0(i)}, \quad n\geq 1, \quad i,j\in \bbZ/p\bbZ,
\eeq
and starting from state $0$.
Therefore, we may write
\beq
\ba
&\bP(\gs > n,\ S_n\in \tau) \ge \bP(\gs > n,\ S_n\in \tau^{(0)}) \\
&= e^{-\phi(\gb)n}
\sum_{m=1}^n \sum_{0<u_1<\ldots<u_m=n} \sumtwo{x_1, \ldots, x_{m-1}}{x_0 = x_m = 0} \prod_{i=1}^m \Big( q_{x_{i-1}, x_i}(u_i - u_{i-1}) e^{-\gb+\phi(\gb)(u_i - u_{i-1})} \Big)\\
&= e^{-\phi(\gb)n}
\sum_{m=1}^n \sum_{0<u_1<\ldots<u_m=n} \sumtwo{x_1, \ldots, x_{m-1}}{x_0 = x_m = 0} \prod_{i=1}^m q^\gb_{x_{i-1}, x_i}(u_i - u_{i-1})\\
&= e^{-\phi(\gb)n} \cP_\gb(n\in \rho_0),
\ea 
\eeq
where $\rho_0$ is the subset of $\rho$ formed by the Markov renewal points with state $0$. It turns out that it is a renewal process. By the second inequality in \eqref{eq:comp.phi} (that we have already proven in Step $2$) $q^\gb_{ij}$ decays exponentially in $n$ for all $i,j\in \bbZ/p\bbZ$. This implies (as the modulating Markov chain has finite state space) that the inter-arrival law of $\rho_0$ also decays exponentially in $n$, which implies integrability. Therefore, by the renewal theorem, $\cP_\gb(n\in \rho_0)$ converges to some constant (that is the inverse of the mean inter-arrival time). This concludes this step.\\

{\bf \noindent Step 5. Proof of \eqref{eq:comp.phi}.}
The second inequality has already been established in Step 2, so let us prove the first inequality. A standard coupling argument yields
\beq
\sum_{j} Q_{ij}(\phi) = \bE_{\tau_i}(e^{\phi\theta_1}) \le \bE({e^{\phi\theta_1^{\max}}}), \qquad i\in \bbZ/p\bbZ,
\eeq
where $\theta_1 = \inf\{n\ge 1 \colon S_n \in \tau\}$ and $\theta_1^{\max} = \inf\{n\ge 1 \colon S_n \in t_{\max}\bbZ\}$.
By Proposition \ref{pr:homo}, we get that $\sum_{j} Q_{ij}(\phi(\gb,t_{\max})) = e^{\gb}$. Thanks to Lemma \ref{lem:PFone} below, it means that $\gL(\phi(\gb,t_{\max})) \le e^{\gb}$ and we get the desired bound, as $\Lambda$ is non-decreasing.\\

{\bf \noindent Step 6. Final upper bound.}
We now conclude by removing the condition $\{S_n \in \tau\}$ in the upper bound.
To this end, we decompose according to the last visit to $\tau$ before $n$:
\beq
\bP(\gs > n,S_n \notin \tau) = \sum_{m=0}^{n-1} \sum_{j=0}^{p-1} \bP(\gs > m,\ S_m\in \tau^{(j)}) \bP_{\tau_j}(S_k \notin \tau,\ k\le n-m).
\eeq
By using Proposition \ref{pr:small_ball}, we get that there exists $C$ such that for all $0\le j <p$ and $n\ge 1$,
\beq
\bP_{\tau_j}(S_k \notin \tau,\ k\le n) \le C e^{- \min g(t_{ij}) n} = C e^{- g(t_{\max})n}.
\eeq
By using \eqref{eq:comp.phi}, we get
\beq
\ba
\bP(\gs > n ,S_n \notin \tau ) &\le C \sum_{m=0}^n \sum_{j=0}^{p-1} \bP(\gs > m,\ S_m\in \tau^{(j)}) e^{-\phi(\gb) (n-m)}\\
&= C \sum_{m=0}^n \bP(\gs > m,\ S_m\in \tau) e^{-\phi(\gb) (n-m)}\\
&\le Cn^2p \exp(-\phi(\gb) n),
\ea
\eeq
where we have used \eqref{eq:A15} to go from the second to the last line.

\begin{lemma}\label{lem:PFone}
If the sums over lines of a non-negative matrix $A$ are less than one, then its Perron-Frobenius eigenvalue is less than one.
\end{lemma}

\begin{proof}[Proof of Lemma~\ref{lem:PFone}]
Let $\lambda$ be an eigenvalue of $A$ and $v$ an associated eigenvector such that $v_{i_*}=\max_i v_i >0$. Then 
\beq
\lambda v_{i_*}=(Av)_{i_*}\leq \sum_j A_{i_*,j}v_{i_*} \leq v_{i_*},
\eeq
and that is enough to conclude as  $v_{i_*}>0$.
\end{proof}

\bibliographystyle{plain}
\bibliography{biblio}

%%%%%%%%%%%%%%%%%%%%%%

\end{document}